\documentclass[11pt]{article}

\usepackage{amsmath,amsthm,amssymb,amsbsy,amsfonts,graphics,epsfig,graphicx,mathabx}
\usepackage{color}

\usepackage{etoolbox}
\let\bbordermatrix\bordermatrix
\patchcmd{\bbordermatrix}{8.75}{4.75}{}{}
\patchcmd{\bbordermatrix}{\left(}{\left[}{}{}
\patchcmd{\bbordermatrix}{\right)}{\right]}{}{}

\newcommand{\C}{{\mathbb C}}

\newcommand{\Z}{{\mathbb Z}}
\newcommand{\N}{{\mathbb N}}

\newcommand{\M}{\mathcal{M}}
\newcommand{\Oo}{{\cal O}}

\newcommand{\beq}{\begin{equation}}
\newcommand{\eeq}{\end{equation}}
\newcommand{\bea}{\begin{eqnarray*}}
\newcommand{\eea}{\end{eqnarray*}}
\newcommand{\beqa}{\begin{eqnarray}}
\newcommand{\eeqa}{\end{eqnarray}}
\newcommand{\bmat}{\begin{bmatrix}}
\newcommand{\emat}{\end{bmatrix}}
\newcommand{\ba}{\begin{align}}
\newcommand{\ea}{\end{align}}
\newcommand{\bas}{\begin{align*}}
\newcommand{\eas}{\end{align*}}
\newcommand{\tit}{\textit}

\newcommand{\tr}{\mbox{Tr}}

\newcommand{\dsp}{\displaystyle}

\newcommand{\KMS}{Kac-Murdock-Szeg\H{o}}

\newcommand{\ve}{\varepsilon}

\newtheorem{theo}{{\bf Theorem}}[section]
\newtheorem*{theo*}{{\bf Theorem}}
\newtheorem{cor}[theo]{{\bf Corollary}}
\newtheorem{lem}[theo]{{\bf Lemma}}
\newtheorem{prob}{{\bf Problem}}
\newtheorem{defn}[theo]{{\bf Definition}}

\theoremstyle{definition}

\newtheorem{rem}[theo]{{\bf Remark}}

\newcommand{\bth}{\begin{theo}}
\renewcommand{\eth}{\end{theo}}
\newcommand{\bdefi}{\begin{defn}}
\newcommand{\edefi}{\end{defn}}
\newcommand{\bprob}{\begin{prob}}
\newcommand{\eprob}{\end{prob}}

\numberwithin{equation}{section}

\author{Alain Bourget\thanks{\textbf{Mailing address}: Department of Mathematics, California State University (Fullerton),
McCarthy Hall 154, Fullerton CA 92834 (US). \textbf{Email address}: abourget@fullerton.edu, allenalvarez@csu.fullerton.edu,  tmcmillen@fullerton.edu }, Allen Alvarez Loya\footnotemark[1] \ and Tyler McMillen\footnotemark[1] \footnote{Corresponding author.}}

\title{Spectral Asymptotics for  \KMS \ Matrices}

\begin{document}

\maketitle

\bibliographystyle{plain}
 
\begin{abstract}

Szeg\H{o}'s First Limit Theorem provides the limiting statistical distribution (LSD) of the eigenvalues of large Toeplitz matrices.  Szeg\H{o}'s Second (or Strong)  Limit Theorem for Toeplitz matrices gives a second order correction to the First Limit Theorem, and allows one to calculate asymptotics for the determinants of large Toeplitz matrices.  In this paper we survey results extending the first and strong limit theorems to Kac-Murdock-Szeg\H{o} (KMS) matrices.  These are matrices whose entries along the diagonals are not necessarily constants, but modeled by functions.   We clarify and extend some existing results, and explain some apparently contradictory results in the literature. 
\end{abstract}

\bigskip

\bigskip


\newpage
\tableofcontents

\section{Introduction}
\subsection{The problem and its history}

For any  function $a$ of two variables with Fourier series
\begin{equation} \label{FS1}
a(x,t) =\sum_{k \in \Z} \hat{a}_k(x) e^{ikt},
\end{equation}
Kac, Murdock and Szeg\H{o} \cite{kamusz53, grsz58} introduced in 1953 what they called \tit{generalized Toeplitz} matrices   as the matrices of the form
\beq 
T_n(a) = \left[ \hat{a}_{j-i} \left(\frac{i+j}{2n+2}\right) \right]_{i,j=0}^{n}
\label{KMStype}
\eeq
We will call matrices of this form \KMS \ (KMS) matrices, although this is not the generally accepted term, as we will explain below.  Such matrices have sometimes been called \tit{generalized Toeplitz}, \tit{Toeplitz-like}, \tit{twisted Toeplitz}, \tit{variable coefficient Toeplitz} matrices, and probably some other terms which we have yet to run across.  In addition,  Tilli \cite{ti98b},  motivated by applications to differential equations, introduced what he called \tit{locally Toeplitz} matrices, which are closely related to  a special class of KMS matrices (see more below).     As the above mentioned terms are somewhat ambiguous, we prefer the term KMS.  

Note that when $a(x,t)=a(t)$ is independent of $x$, $T_n(a)$ is Toeplitz.  Toeplitz matrices can be characterized by the condition that the differences between elements on the $k$th diagonal are zero:
\[T_n(a)_{j+1,j+k+1}  - T_n(a)_{j,j+k} = 0 
\] 
for all $j,k$ where the terms are defined.  KMS matrices have the property that the differences between elements on the $k$th diagonal approach zero as $n\rightarrow\infty$:
\begin{align*} 
T_n(a)_{j+1,j+k+1} - T_n(a)_{j,j+k} 
 & = \hat{a}_{k}\left(\frac{2j+k}{2n+2} + \frac{1}{n+1}\right) -  \hat{a}_{k}\left(\frac{2j+k}{2n+2}\right) 
\\
& = o(1)
\end{align*}
as long as the $\hat{a}_k$ are continuous.  KMS matrices are thus natural generalizations of Toeplitz matrices, and they are `locally Toeplitz' in this sense.  As long as the functions $\hat{a}_k$ are continuous, or even Riemann integrable, locally they are not too far from being Toeplitz.  For this reason, much of the machinery used in the study of Toeplitz matrices can be applied to KMS matrices, as long as one is careful to keep track of the error terms.

\begin{rem}
The definition \eqref{KMStype} has the advantage that $T_n(a)$ will be Hermitian exactly when $a(x,t)$ is real valued.  However, in many applications, the indexing is not always so neat.
The indexing in \eqref{KMStype} of $(i+j)/(2n+2)$ along the $(j-i)$th diagonal is not necessary for the First Theorem \eqref{First KMS} on eigenvalue distributions to hold.  Sampling the functions $\hat{a}_k$ along the diagonals at any partition whose mesh size approaches zero will lead to the same \emph{limiting statistical distribution (LSD)} of the eigenvalues (see \S\ref{partition}).
  Indeed,
Kac, Murdock and Szeg\H{o} also considered matrices of the form
\[  \left[ \hat{a}_{j-i} \left(\frac{\min\{i,j\}}{n+1}\right) \right]_{i,j=0}^{n}
\quad \text{ and } \quad
 \left[ \hat{a}_{j-i} \left(\frac{\max\{i,j\}}{n+1}\right) \right]_{i,j=0}^{n}
 \]
These matrices, as the one in \eqref{KMStype},  will also be Hermitian if and only if the symbol $a$ is real-valued.  And, they have the same LSD as those defined by \eqref{KMStype}.
However, while the asymptotic eigenvalue distribution is independent of how the indexing is done,
 the asymptotics of the determinants of $T_n(a)$ are extremely sensitive to the way in which the indexing is done (see \S\S\ref{DSOsection} and Remark~\ref{shiftremark}).  
\end{rem}

The main result of the original paper of Kac, Murdock and Szeg\H{o} is concerned with proving a generalized First Szeg\H{o}'s Limit Theorem for matrices of the type \eqref{KMStype}.
In that paper  they also included a chapter on extreme eigenvalues of Toeplitz matrices of the form 
$$\left[\rho^{|j-i|}\right]_{i,j=0}^n, 
$$
where $0<\rho < 1$.  It is this kind of Toeplitz matrix that has often been referred to as a KMS matrix in the literature.  In particular,  in a series of papers, Trench \cite{tr02} generalized this notion to a broader class of Toeplitz matrices and obtained asymptotic results for the spectra in some cases where the First Szeg\H{o}'s Limit Theorem does not apply.  However, it is, in our opinion, a misnomer to refer to these as KMS matrices, when they are, after all, Toeplitz matrices.


\smallskip

KMS and related matrices have received  a lot of attention lately with applications to such various fields as statistical mechanics \cite{ka13,demc98}, differential equations \cite{bosh08, ti98b, ti99b, se00, mcboag09, shtata12, comase16}, quantum integrable systems \cite{agbo08, bomc09, bo10} and the Heisenberg group \cite{budihiwi16, budihiwi16b}, among others.     There has been a renewed interest in matrices of this type which followed a renewed interest in Toeplitz forms and their connections to orthogonal polynomials.  The history of citations of \cite{kamusz53} illustrates the general trend.  MathSciNet currently lists 77 citations for \cite{kamusz53}.  Seven of these were from papers published from 1958 to 1965.  There were no citations between the years 1965 and 1999.  After these 34 years of dormancy, from 1999 on there have been 70 citations.  Curiously, only a few of those papers citing \cite{kamusz53} dealt with matrices of the form \eqref{KMStype}.   Google Scholar, which keeps a more comprehensive collection of citations including from math and non-math journals, lists 209 citations for \cite{kamusz53}.  Google Scholar shows more citations in the years 1965 to 1999, but the same general trend of relatively few citations until near the end of the 20th century, followed by renewed interest at the turn of the century.  Figure~\ref{cites} shows the number of citations per year from each source.

\smallskip

\begin{figure}[h]
\begin{center}
\includegraphics*[height=1.8in]{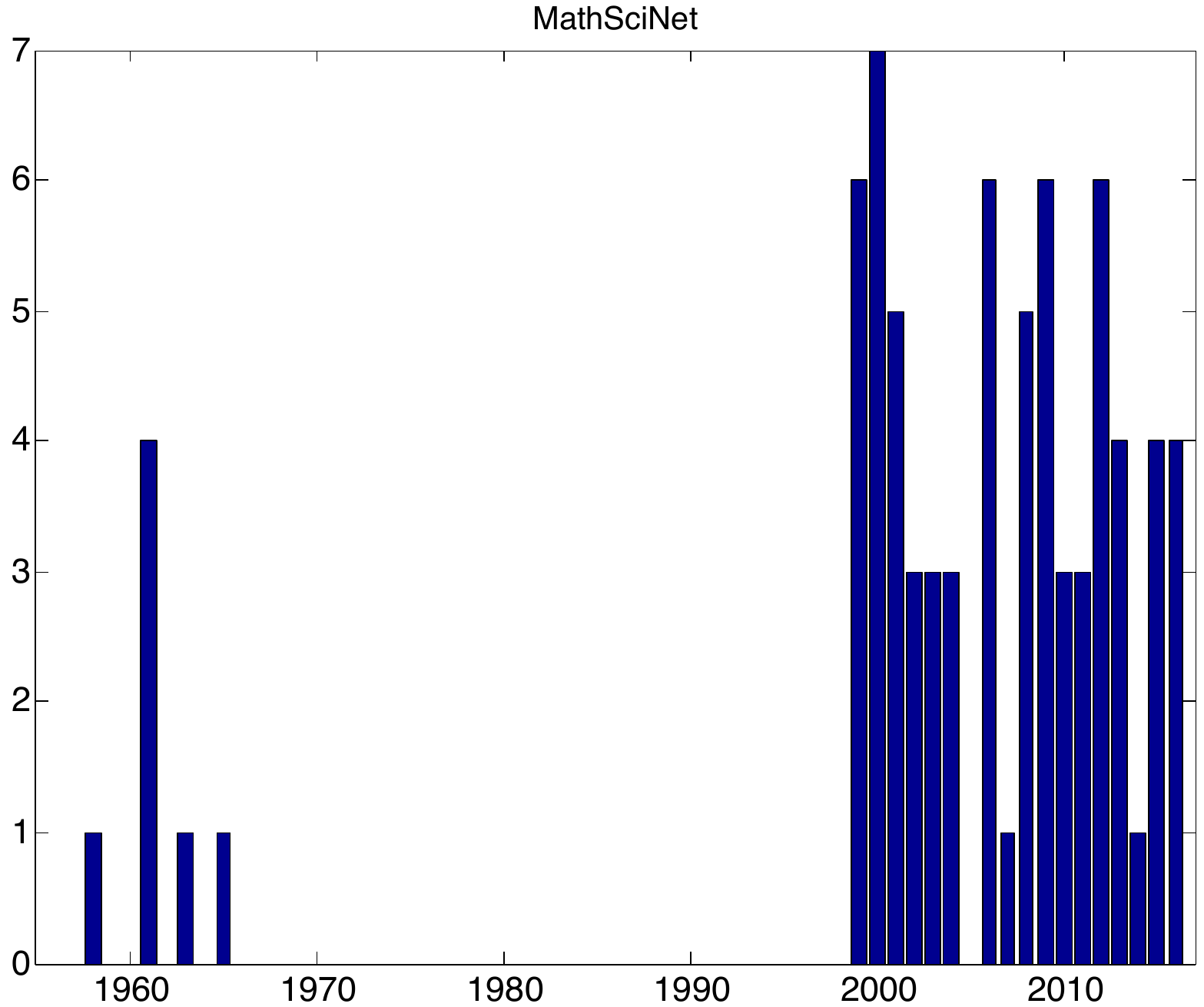}
\hskip .1in
\includegraphics*[height=1.8in]{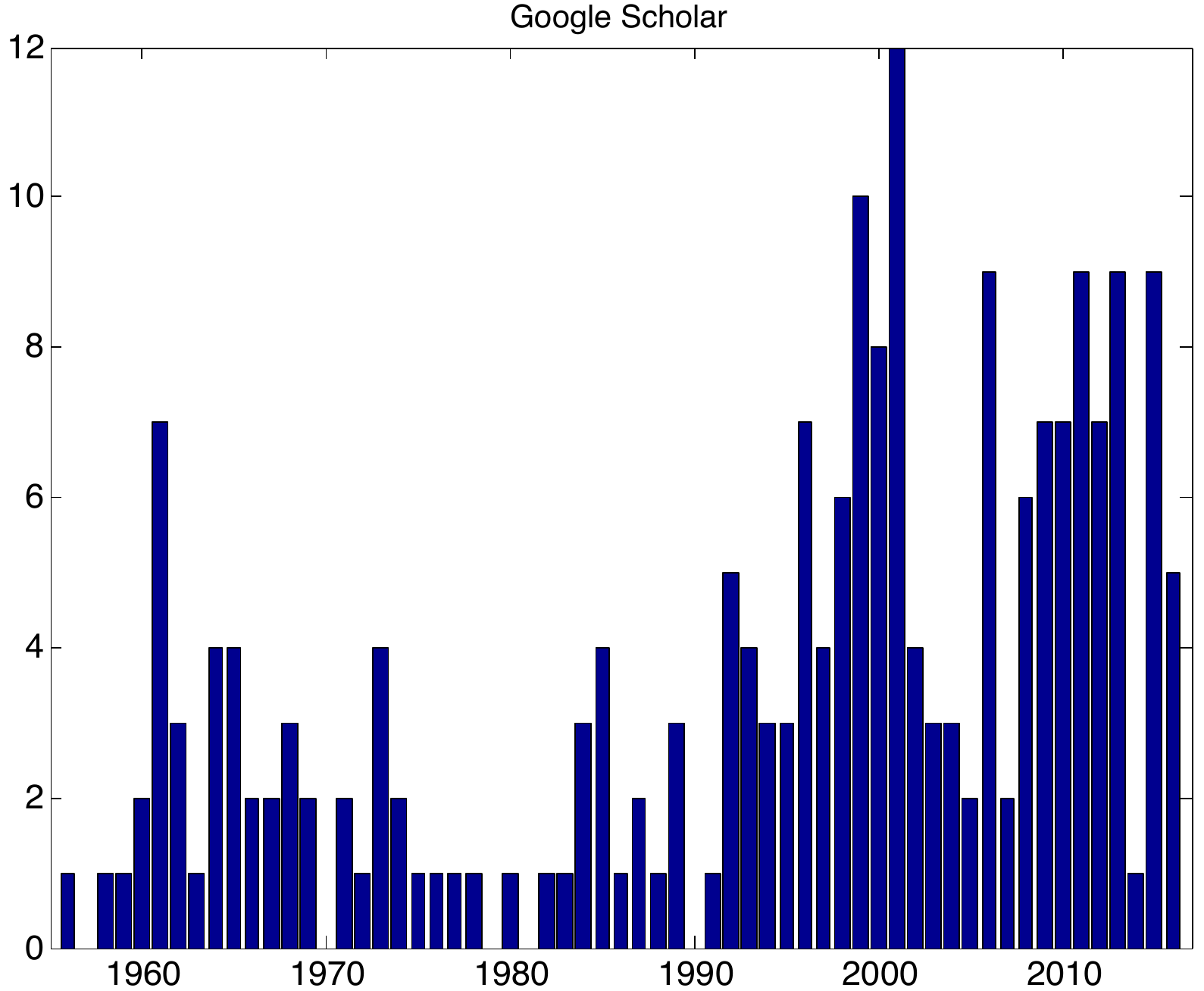}
\end{center}
\caption{Citations of \cite{kamusz53} by year, according to MathSciNet (left) and Google Scholar (right).}
\label{cites}
\end{figure}

In the early 1960's, Mejlbo and Schmidt \cite{mesc61, mesc62} proved a Strong Szeg\H{o}'s Theorem for KMS matrices, which we discuss at length below and in \S\ref{sec.strong}.  After the mid-1960's, interest in such matrices waned.  One notable exception is the paper of Trotter \cite{tr84}, who in 1984 generalized the results of Kac, Murdock and Szeg\H{o} to symbols in which the functions $\hat{a}_k$ are only Riemann integrable.
Interest in spectral asymptotics of KMS matrices and their like renewed toward the end of the 20th century.  In 1998, Shao  \cite{sh98} generalized the First Limit Theorem to symbols of bounded variation.  That same year, Deift and McLaughlin \cite{demc98} considered sequences of Jacobi matrices with variable coefficients in the context of the Toda lattice.  Also in 1998, Tilli \cite{ti98b} introduced locally Toeplitz matrices.  Tilli originally considered matrices that arise in the discretization of certain differential equations, which are, modulo finite rank perturbations, asymptotic to KMS matrices with symbols $a(x,t)=f(x)g(t)$.  He derived a First Szeg\H{o}'s Theorem for such matrices, in effect rederiving the result of Kac, Murdock and Szeg\H{o} for a special case.    This notion of locally Toeplitz has been generalized by several authors and applied quite fruitfully to the study of differential equations, notably by Serra-Capizzano.  (See  \cite{se03,se06}, and \cite{gase16} for a  review of recent work on locally Toeplitz matrices and applications to differential equations.)

\smallskip 

In 1999, Kuijlaars and Van Assche \cite{kuva99} published a paper on the asymptotic density of zeros of orthogonal polynomials with continuously varying coefficients.  Such zeros  are  eigenvalues of symmetric tridiagonal KMS matrices whose diagonal entries are modeled by continuous functions.  This has become a highly cited and influential paper.  This was followed by a paper by Kuijlaars and Serra Capizzano \cite{kuse01} that extended the result to diagonals modeled by Riemann integrable functions.  
Kuijlaars et.al. obtained a formula for the asymptotic density of eigenvalues, whereas Kac, Murdock and Szeg\H{o} calculated the LSD via an integral (eqn \ref{First KMS} below).  However, both results give essentially the same information, and it is a trivial exercise to extract the density from the LSD in the Jacobi case.  The proof of Kuijlaars and Van Assche used powerful tools from potential theory, and were completely different from the moments method employed by Kac, Murdock and Szeg\H{o}.  Still, their results were essentially a special case of results obtained nearly 50 years earlier using more elementary tools.

\smallskip

It is apparent that the results of Kac, Murdock and Szeg\H{o} were not well known in the mathematical community.    This includes, until quite recently, the authors of the present paper, who undertook to write this article partly in order to rectify this situation.  Indeed, the first author would have benefitted greatly from the knowledge of \cite{kamusz53} when writing the articles \cite{agbo08, agbo09, bo12} on generalizations of tridiagonal KMS matrices. 
The main purpose of the present paper is to collect and clarify results on the statistical distribution of eigenvalues of KMS matrices as $n\rightarrow\infty$ (the limiting statistical distribution, or LSD).  We will be able to extend some of the known results.  We will also explain some of the results in the literature that appear to be contradictory.  
The paper is mostly self-contained.  For completeness, we will include most of the proofs.  We have endeavored to present the proofs that are the clearest and most powerful.

\subsection{The First and Second Szeg\H{o}'s Limit Theorems}

Suppose the symbol $a$ in \eqref{FS1} is real-valued, so that $T_n(a)$ is Hermitian.  Suppose, further, that the functions $\hat{a}_k$ are continuous on $[0,1]$ and satisfy the decay condition
\begin{equation} \label{decay}
 \vvvert a \vvvert := \sum_{k \in \Z} \|\hat{a}_k\|_\infty < \infty.
\end{equation}
By Gershgorin's disks Theorem, the spectrum of $T_n(a)$ lies inside the closed interval $[-\vvvert a \vvvert,   \vvvert a \vvvert ] $. 
Under these condition, Kac, Murdock and Szeg\H{o} \cite{kamusz53} proved the following First Limit Theorem for KMS matrices in 1953:
\begin{equation} \label{First KMS}
  \text{Tr}[\varphi(T_n(a))] = \frac{n+1}{2\pi} \int_0^1 \int_0^{2\pi} \varphi(a(x,t)) \, dt \, dx + o(n) \qquad (n \to \infty)
\end{equation}  
for any continuous function $\varphi$ on  $ [-\vvvert a \vvvert, \vvvert a \vvvert]$.  When the symbol $a$ is independent of $x$, i.e. $a(x,t)=a(t)$, the result in \eqref{First KMS} reduces to the First Szeg\H{o}'s Limit Theorem for Toeplitz matrices. 

\smallskip

The formula \eqref{First KMS} can be written in the equivalent form
\[ \lim_{n\rightarrow\infty} \frac{\sum_{j=1}^{n+1}\varphi\left(\lambda_j(T_n(a))\right) }{n+1} = \frac{1}{2\pi} \int_0^1 \int_0^{2\pi} \varphi(a(x,t)) \, dt \, dx
\]
where $\lambda_j(T_n(a)) \ (j=1,\dots, n+1)$ are the eigenvalues of $T_n(a)$.   
The result \eqref{First KMS} thus gives us the LSD of the eigenvalues of $T_n(a)$.  It says, roughly, that as $n \to \infty$, the eigenvalues of $T_n(a)$ distribute like the values of $a(x,t)$ sampled at regularly spaced points in the rectangle $0\leq t\leq 2\pi, \ 0\leq x\leq 1$.  Alternatively, \eqref{First KMS} is equivalent to the following.  Let $\alpha < \beta$ and let $N(n;\alpha,\beta)$ be the number of eigenvalues $\lambda$ of $T_n(a)$ satisfying $\alpha\leq \lambda \leq \beta$.  Then 
\[ \lim_{n\rightarrow\infty} \frac{N(n;\alpha,\beta)}{n+1} = \frac{\sigma}{2\pi}
\]
where $\sigma$ is the area of the sub-domain in the rectangle $0\leq t\leq 2\pi, \ 0\leq x\leq 1$ such that $\alpha\leq a(x,t)\leq\beta$.

\smallskip

If one makes the additional assumptions $a(x,t) >0$ and  $\log a(x,t) \in L^1([0,1] \times [0,2\pi])$, then \eqref{First KMS} for $\varphi(x) = \log x$ reads as
\begin{equation*}
 \log \det T_n(a) =  \frac{n+1}{2\pi} \int_0^1 \int_0^{2\pi} \log(a(x,t)) \, dt \, dx  + o(n) 
 \end{equation*}
When $a(x,t)$ is independent of $x$, the Szeg\H{o}'s Second--or Strong--Limit Theorem gives the more precise statement  \cite{grsz58, bogr05}:
 \begin{equation} 
 \label{Strong Szego1}
  \log  \det (T_n(a)) = \frac{n+1}{2\pi}  \int_0^{2\pi} \log(a(t)) \, dt   + E(a) + o(1)
\end{equation} 
where $E(a)$ is the constant 
\[E(a)=  \sum_{k =1}^{\infty} k  \widehat{(\log a)}_k\widehat{(\log a)}_{-k}
\]
and 
\[\widehat{(\log a)}_k = \frac{1}{2\pi}\int_0^{2\pi} \log(a(t)) e^{-ikt} dt
\]
denotes the $k$th Fourier coefficient of $\log a$.   The Strong Szeg\H{o}'s Theorem \eqref{Strong Szego1}  is more often expressed in the form
\beq 
\lim_{n \to \infty} \frac{\det T_n(a)}{ G(a)^{n+1}} = \exp\left[ E(a) \right]
\label{SS2}
\eeq
where 
\[G(a) = \exp\left[ \frac{1}{2\pi} \int_0^{2\pi} \log a(t) \, dt\right] 
\]
is the geometric mean of $a$.


In the early 60's, Mejlbo and Schmidt \cite{mesc61, mesc62}  gave the following extension of Szeg\H{o}'s Strong Theorem to KMS matrices. Suppose $a$ is a complex-valued symbol such that $\hat{a}_k \in C^2([0,1])$ and satisfy the conditions 
\begin{equation*} \label{MSC}
\sum_{k \in \Z} \| \hat{a}_k\|_\infty<1, \  \sum_{k \in \Z} \|\hat{a}'_k\|_\infty < \infty, \   \sum_{k \in \Z} \|\hat{a}^{''}_k\|_\infty < \infty, \text{ and } \sum_{k \in \Z} |k|^\alpha \|\hat{a}_k\|_\infty < \infty
\end{equation*}
for some $\alpha>2$. Suppose, further, that $a$ satisfies the diagonal dominance condition $\min |\hat{a}_0| > \sum_{ k \neq 0} \| \hat{a}_k\|_\infty$.  Then, we have
\begin{eqnarray} \label{Strong KMS}
 \lim_{n \to \infty}  \frac{\det T_n(a)}{G(a)^{n+1}}   = \exp \frac{1}{2}\{  e(a;0)-e(a;1) +E(a;0)+E(a;1) \} 
\end{eqnarray}
where $G(a)$, $e(a;x)$ and $E(a;x)$ are now given by 
\beq 
G(a) = \exp\left[\frac{1}{2\pi} \int_0^1 \int_0^{2\pi} \log(a(x,t)) \, dt \, dx \right], 
\label{GeE1}
\eeq
and
\begin{align} \label{GeE2}
e(a;x) = \frac{1}{2\pi} \int_0^{2\pi} \log a(x,t) \, dt, 
\end{align}
\begin{align}
E(a;x) =   \sum_{k =1}^{\infty} k  \, \widehat{(\log a)}_k(x) \, \widehat{(\log a)}_{-k}(x).
\label{GeE3}
\end{align}

Of course, when $a$ does not depend on $x$, $E(a;x) = E(a)$, and Mejlbo and Schmidt's result coincides with the Szeg\H{o} Strong Theorem. Quite remarkably, the error term obtained by Mejlbo and Schmidt in \eqref{Strong KMS} depends only on the values of the symbol $a$ at the endpoints  $x=0$ and $x=1$.

%
\subsection{Content of the present paper}

In the next two sections,  we prove a First Limit Theorem and a Strong Limit Theorem for KMS matrices under some minor improvements of the results mentioned above. For the sake of consistency and simplicity,  both proofs use the moments method  rather than the probabilistic method of Trotter \cite{tr84} in the First Limit Theorem or the operator method of Ehrhardt and Shao \cite{ehsh01} for the Strong Theorem. \S\ref{sec.first} deals with the First Limit Theorem.  In addition to proving this result for KMS matrices, we include a section on the asymptotic distribution of singular values.  When the matrices $T_n(a)$ are normal, we obtain the LSD of the eigenvalues.  When $T_n(a)$ is not normal, we cannot obtain the LSD.  However, we can say that most of the eigenvalues will accumulate in the extended range of the symbol $a$.  This is the content of \S\ref{sec.clustering}.

\smallskip

\S\ref{sec.strong} deals with the Strong Limit Theorem.  We prove this theorem along the same lines as Mejlbo and Schmidt.  We then explain the results of Ehrhardt and Shao, who proved a similar theorem, but obtained a different formula from that of Mejlbo and Schmidt.  We explain where the difference arises, and why their methods cannot be applied to KMS matrices.  The last subsections of \S\ref{sec.strong} deal with the discrete Schr\"{o}dinger operator, i.e the special KMS matrix ``discrete Laplacian $+$ diagonal'' of the form
\[ -\left[ \delta_{j,j+1} + \delta_{j+1,j} \right]  + \mbox{diag}\left( f(1/n), f(2/n), \dots, f(n/n)\right)
\]  
We report here on recent results for the asymptotics of determinants of such matrices.  These results are extensions of a beautiful result of Kac from 1969 \cite{ka69}.  One such extension is that if one shifts the indexing, one can adjust the limit of the determinant without changing the LSD of the eigenvalues.  Another extension deals with asymptotics when $f$ has jump discontinuities.  In that case, the determinant approaches a value (modulo $o(1)$ terms) that depends on $n$ in a peculiar way.  These results demonstrate the extreme sensitivity of the determinant to small changes, and why the limit of $\det(T_n(a))/G(a)^{n+1}$ cannot exist when the symbol has discontinuities.

We conclude with some open problems and conjectures.  

\section{First Limit Theorems}
\label{sec.first}

We present several First Limit Theorems for sequences of normal and non-normal KMS matrices.  Our presentation follows the original approach of Kac, Murdock and Szeg\H{o}  for some obvious reasons.  First,  their argument is in our opinion the most elegant and simplest one. Second, it allows straightforward generalizations to sequences of matrices with alternative indexing (see \S\ref{partition}) and to sequences of block matrices (see \S\ref{block}). Their arguments can also be easily modified to obtain the LSD of the singular values  (see \S\ref{singular}). We end in \S\ref{sec.clustering} with a clustering property for the spectrum of $T_n(a)$.  

\smallskip

\subsection{Normal matrices}
\label{normal}

Let $a(x,t)$ be a complex-valued function on $[0,1] \times [0,2\pi]$ and let $a_t$, $a_x$ be the functions defined by $a_t(x)=a(x,t)=a_x(t)$. Throughout \S\ref{normal}-\S\ref{sec.clustering}, we assume $a_t$ is Riemann integrable on $[0,1]$ and $a_x \in L^\infty ([0,2\pi])$.  As a consequence of Gershgorin's disks Theorem and condition \eqref{decay}, the spectrum of every $T_n(a)$ lies inside the disk $D_a=\{z \in \C : |z|  \leq \vvvert a \vvvert\}$.

Our first result is concerned with the LSD of sequences of normal KMS matrices. This was first obtained by Trotter \cite{tr84} using some non-trivial probability argument. We give here a more elementary proof based on Kac-Murdock-Szeg\H{o} approach.

\begin{theo} \label{normal}
Let $\{ T_n(a) \}$ be a sequence of normal KMS matrices with symbol $a$ as in \eqref{FS1} and for which \eqref{decay} holds. Then,
\begin{equation}  \label{KMS1eq}
    \lim_{n \to \infty} \frac{\text{Tr}[\varphi(T_n(a))] }{n+1} = \frac{1}{2\pi} \int_0^1 \int_0^{2\pi} \varphi(a(x,t)) \, dt \, dx
\end{equation}
for any $\varphi \in C(D_a)$.
\end{theo}

\begin{proof}  The normality of $T_n(a)$ ensures that 
$$ \tr[ T_n^p T_n^{*q} ] = \sum_{j=1}^{n+1} \lambda_j^p (T_n(a))  \, \overline{\lambda_j^q(T_n(a))} \qquad (p,q \in \N).$$
By the Stone-Weierstrass Theorem, the polynomials in $z$ and $\bar{z}$ are dense in the space of continuous functions on $D_a$.  Consequently, it suffices to prove the result when $\varphi(z)=z^p \bar{z}^q$ with $p, q \in \N$. For simplicity, we write $a_{ij}$ for the entries of $T_n(a)$, and $T_n$ for $T_n(a)$. We have 
\beq 
\label{trace 1}
    \tr[ T_n^p T_n^{*q} ]  =  \sum_{i=0}^{n} \sum_{j_1,...,j_{p+q}=0}^n a_{i,j_1} a_{j_1,j_2} \cdots a_{j_{p-1},j_p} \bar{a}_{j_{p+1},j_p} \cdots \bar{a}_{i, j_{p+q}}.
\eeq     
Let $h=(h_1,...,h_p)$ and $k=(k_1,...,k_q)$ be the multi-indices defined by 
$$h_1=j_1-i, \ h_2=j_2-j_1, \ldots, h_\alpha=j_p - j_{p-1}$$ 
and 
$$k_1=j_{p+1}-j_p, \ldots, k_{q-1}= j_{p+q}-j_{p+q-1}, \ k_q= j-i_{p+q}.$$ 
Adding the above relations, we see that $h$ and $k$ must satisfy the condition $|h|=|k|$. Hence, we can rewrite \eqref{trace 1} as
\beq
\label{trace 2}
    \text{Tr}[ T_n^r T_n^{*s}  ]  =  \sum_{i=0}^{n} \sideset{}{'}\sum_{|h|=|k|}  \prod_{l=1}^p \hat{a}_{h_l} \left( \frac{2i+\nu_l}{2n+2} \right) \, \prod_{m=1}^q \bar{\hat{a}}_{k_m} \left( \frac{2i+\nu_{p+m}}{2n+2} \right) 
\eeq
where 
$$ \nu_j = \begin{cases}
                     2h_1+2h_2+\cdots + 2h_{j-1}+h_j  & \text{ if } 1 \leq j \leq p \\
                     \nu_p -2k_1 -2k_2 - \cdots - 2k_{j-1} - k_j & \text{ if } p+1 \leq j \leq p+q.
                \end{cases}$$  
The prime on the middle sum of \eqref{trace 2} is used to indicate that only multi-indices  satisfying $ -2i \leq \nu_j \leq n-2i$ must enter into the summation.  

Let $\varepsilon>0$. By condition \eqref{decay}, there exists $r_0$ large enough so that  
$$ \sum_{|k| >r} \| \hat{a}_k \|_\infty < \varepsilon \qquad (\forall r>r_0).$$ 
Pick  $r>r_0$ and let $M(h,k)=\max\{|h_l|,|k_m|\}$. It follows that 
\begin{align}
  &  \sum_{i=0}^{n} \sideset{}{'}\sum_{ \substack{ |h|=|k| \\M(h,k) >r}}  \prod_{l=1}^p \hat{a}_{h_l} \left( \frac{2i+\nu_l}{2n+2} \right) \, \prod_{m=1}^q \bar{\hat{a}}_{k_m} \left( \frac{2i+\nu_{p+m}}{2n+2} \right)   \nonumber \\
     & \qquad \qquad  \leq \sum_{i=0}^n \left( \sum_{|k|>r} \|\hat{a}_k \|_\infty \right) \ \left( \sum_{k \in \Z} \|\hat{a}_k\|_\infty \right)^{p+q-1} \nonumber \\
     & \qquad \qquad \leq (n+1) \, \vvvert a \vvvert^{p+q-1} \ \varepsilon,
\end{align}
i.e. the above sum is $o(n)$. Hence, it suffices to consider in \eqref{trace 2}   multi-indices $h$ and $k$ for which $|h|=|k|$ and $M(h,k) \leq r$. For such multi-indices, we have
\begin{align*}
    \sum_{|h|=|k|} & =   \sideset{}{'}\sum_{|h|=|k|} +  \sum_{ \substack{|h| = |k| \\ \min \nu_j <-2i \\ \max \nu_j <n-2i}} + \sum_{\substack{|h|=|k| \\ \max \nu_j > n-2i}} \\
      & =  \sideset{}{'}\sum_{|h|=|k|} +\,  o(n)
 \end{align*}   
since $M(h,k) \leq r$. Hence, we can replace \eqref{trace 2} by
\begin{equation*} 
    \text{Tr}[ T_n^p T_n^{*q} ]  =  \sum_{i=0}^{n} \sum_{ \substack{|h|=|k| \\ M(h,k) \leq r}}  \prod_{l=1}^p \hat{a}_{h_l} \left( \frac{2i+\nu_l}{2n+2} \right) \, \prod_{m=1}^q \bar{\hat{a}}_{k_m} \left( \frac{2i+\nu_{p+m}}{2n+2} \right) +o(n).
\end{equation*}

We now use the fact that the functions $\hat{a}_k$ are Riemann integrable on $[0,1]$ and $|\nu_j| =\Oo(1)$, so we can shift - up to an error of $o(n)$ - each of the expressions inside the functions $\hat{a}_k$  to obtain
\begin{align*} 
    \text{Tr}[ T_n^p T_n^{*q} ] & =  \sum_{i=0}^{n} \sum_{ \substack{|h|=|k| \\ M(h,k) \leq r}} \prod_{l=1}^p \hat{a}_{h_l} \left( \frac{i}{n+1} \right) \, \prod_{m=1}^q \bar{\hat{a}}_{k_m} \left( \frac{i}{n+1} \right)  +o(n)  \\
     & =   (n+1) \int_0^1 \sum_{ \substack{|h|=|k| \\ M(h,k) \leq r}} \prod_{l=1}^p \hat{a}_{h_l} \left( x \right) \, \prod_{m=1}^q \bar{\hat{a}}_{k_m}  (x) \, dx + o(n) 
     \label{riemsum} \\
     & =  \frac{n+1}{2\pi} \int_0^1 \int_0^{2\pi} \ a^p_r (x,t) \ \overline{a^q_r (x,t)} \, dt \, dx +o(n)
\end{align*}
where $a_r$ is the trigonometric polynomial given by
$$ a_r(x,t)  = \sum_{|k| \leq r} \hat{a}_k(x) \, e^{ikt}.$$
By condition \eqref{decay}, $a_r$ converges uniformly to $a$ on $[0,1] \times [0,2\pi]$ as $r \to \infty$. Therefore, we conclude
\begin{align*}
  \text{Tr}[ T_n^p T_n^{*q} ] =  \frac{n+1}{2\pi} \int_0^1 \int_0^{2\pi} \ a^p (x,t) \ \overline{a^q (x,t)} \, dt \, dx +o(n)
\end{align*}
as desired.   
\end{proof}

Obviously, the above argument breaks down when $T_n(a)$ is not normal, and hence we are unable to compute the LSD of arbitrary sequences of KMS matrices. However, we can easily prove the following weaker result. 

\begin{theo} \label{First Analytic}
Let $\{T_n(a)\}$ be a sequence of KMS matrices that satisfies condition \eqref{decay}. Then, we have
\begin{equation} \label{analyticeq1}
    \lim_{n \to \infty} \frac{ \text{Tr}[\varphi(T_n(a))] }{n+1}= \frac{1}{2\pi} \int_0^1 \int_0^{2\pi} \varphi(a(x,t)) \, dt \, dx
\end{equation}
for any analytic function $\varphi$ on the closed disk $D_{a}$.
\end{theo}

\begin{proof}
From the proof of Theorem \ref{normal} with $q=0$, we deduce that 
$$   \text{Tr}[ T_n^p(a) ] =  \frac{n+1}{2\pi} \int_0^1 \int_0^{2\pi} \ a^p (x,t) \, dt \, dx +o(n)$$
for every $p \in \N$. Note that the normality is not needed in this case. The conclusion then follows from Mergelyan's Theorem that asserts that polynomials are dense in the space of analytic functions on $D_a$.  
\end{proof}

\subsection{Alternative indexing}
\label{partition}

We start with a result that should be fairly obvious from the proof of Theorem ~\ref{normal}.  There we used the indexing by $(i+j)/(2n+2)$ along the $(j-i)$th diagonal to form a Riemann sum.  But, any partition whose mesh size approaches zero will accomplish the same, so the strict indexing in the definition \eqref{KMStype} is not necessary for Theorem \ref{normal} to hold.  Indeed, any partition will do.  The partition points can also be shifted by finite amounts.  The following theorems, whose proofs are omitted, should thus be intuitive.  We refer the reader to \cite[Cf. Theorem~3.6 and Corollary~3.7]{bomc15} for detailed proofs of some of the results.

\smallskip

\begin{defn} 
Let $\{ \mathcal{P}_n \}$ be a sequence of partitions of $[0,1]$ such that
$$ \mathcal{P}_n=\left\{x_{0}^{(n)},x_{1}^{(n)},...,x_{n+1}^{(n)} \right\}$$ and whose meshes satisfy $\| \mathcal{P}_n \|= o(1)$ as $n \to \infty$. To every $\mathcal{P}_n$, we associate $(2n+1)$-tagged partitions 
$\{\mathcal{P}_{n;k}\}$ with
$$\mathcal{P}_{n;k} = \left\{ [x_j^{(n)},x_{j+1}^{(n)}] ; \xi_{j;k}^{(n)} \right\} \qquad (-n \leq k \leq n) . $$ 
We defined generalized KMS matrices as the matrices
$$ T_n(a)  = \bigg[ \hat{a}_{j-i}  \left( \xi_{i \wedge j;j-i}^{(n)} \right) \bigg]_{i,j=0}^n.$$
\end{defn}

For instance, it is not hard to see that KMS matrices and their variant (see Remark 1.1) can all be expressed in that form. 

\smallskip
 
The next results are straightforward extensions of Theorem \ref{normal} to sequences of generalized normal KMS matrices. They also extend previous results of Kuiljaars, Van Assche and Serra Capizzano \cite{kuva99,kuse01} for tridiagonal matrices.

 \begin{theo}  \label{altindex}
Let $\{T_n(a)\}$ be a sequence of generalized normal KMS matrices. Then,
for any $\varphi \in C(D_{a})$, we have
\begin{equation*}
   \lim_{n \to \infty} \frac{ \text{Tr}\left[\varphi(T_n(a))\right] }{n+1}  
   =  \frac{1}{2\pi}  \int_0^1 \int_{0}^{2\pi} \varphi( a(x,t))  \, dt \, dx.
\end{equation*}
In the non-normal case, $\varphi$ must be chosen to be analytic in $D_a$. 
\end{theo}

\smallskip

It is possible to give a  slightly more general result for small perturbations of $T_n(a)$. Indeed, let $\{A_n\}$ and $\{B_n\}$ be two sequences of normal matrices whose entries satisfy the estimate
$$ |a_{ij}^{(n)} - b_{ij}^{(n)} | = o \left(  n^{-1/2} \right).$$
The Wielandt-Hoffman inequality \cite{bh97} then implies 
$$ \frac{1}{n+1} \sum_{k=1}^{n+1} | \lambda_k(A_n) - \lambda_k(B_n) | \leq \frac{1}{\sqrt{n+1}} \| A_n-B_n\|_F  = o\left(1  \right).$$
In addition, if we assume that $\{A_n\}$ and $\{B_n\}$ satisfy a decay condition as \eqref{decay}, then $\{A_n\}$ and $\{B_n\}$ have the same LSD.

\begin{theo}  \label{altindex2}
Let $\left\{A_n \right\}$ with $A_n=\left[ a_{ij}^{(n)}\right]_{i,j=0}^n$ be a sequence of normal matrices that satisfies the decay condition
\[ \alpha  =\sup_n \left[ \sum_{k=0}^n \max_{0\leq j\leq n-k} \left|a_{j+k,j}^{(n)} \right| + \sum_{k=1}^n  \max_{0 \leq j \leq n-k}  \left|a_{j,j+k}^{(n)} \right|  \right] < \infty
\]
Suppose there exists a sequence $\{T_n(a)\}$ of generalized normal KMS matrices for which \eqref{decay} holds and 
\[  \left|a_{ij}^{(n)} - \hat{a}_{j-i}\left(\xi_{i \wedge j;j-i}^{(n)} \right)\right| = o\left( n^{-1/2} \right).
\]
If $M=\max\{ \alpha, \vvvert a \vvvert\}$,  then we have
\begin{equation*}
   \lim_{n \to \infty} \frac{ \text{Tr}\left[\varphi(A_n)\right] }{n+1} 
   =  \frac{1}{2\pi} \int_0^1 \int_{0}^{2\pi}  \varphi( a(x,t))  \, dt \, dx
\end{equation*}
for every $\varphi \in C(D_a)$. 
\end{theo}

\subsection{Block matrices}
\label{block}

The proofs of Theorems  \ref{normal} and \ref{First Analytic} -- and their extensions given above --  make no use of the product commutativity of the entries of $T_n(a)$. This naturally suggests to extend our class of symbols to matrix-valued ones.  That is, $a(x,t)$ is now a matrix $A(x,t)$ of some fixed order $k$.  In this context, $T_n(A)$ is then a block matrix of order $k(n+1)$ whose entries $A_{pq}$ are given by the $k \times k$ matrices
$$ A_{pq} = \hat{A}_{q-p} \left( \xi^{(n)}_{p \wedge q;q-p}  \right) = \frac{1}{2\pi} \int_0^{2\pi} A\left( \xi_{p \wedge q;q-p}^{(n)},t \right) \, e^{-i(q-p)t} \, dt.$$   
Obviously, one can also define generalized block KMS matrices by allowing the entries to be evaluated on some tagged partitions of $[0,1]$. Note that our definition includes as special cases the  class of locally Toeplitz matrices \cite{ti98b} and their generalizations \cite{se03}.

As before, we assume $A_t$ is Riemann integrable on $[0,1]$ and $A_x \in L^\infty ([0,2\pi])$ for any given $x$ and $t$.  In the block case, this is equivalent to say that each entry of $A(x,t)$ is Riemann integrable on $[0,1]$ in $x$ and bounded in $t$ on $[0,2\pi]$. Condition \eqref{decay} now reads as 
\begin{equation} \label{decay matrix}
\vvvert A \vvvert = \sum_{k \in \Z}  \ \sup_{x \in [0,1]} \| \hat{A}_k(x) \|_F <\infty.
\end{equation}
Of course, the choice of matrix norm is arbitrary -- for convenience, we picked the Frobenius norm -- any other matrix norm could also be chosen. We can now restate Theorems \ref{normal} and \ref{First Analytic} in the context of generalized block KMS matrices. 

\begin{theo}
Let $\{T_n(A)\}$ be a sequence of normal block generalized KMS matrices that satisfies condition \eqref{decay matrix}. If $A(x,t)$ is normal for all $x$ and $t$, then
\begin{equation*}
    \lim_{n \to \infty} \frac{\text{Tr}[\varphi(T_n(A))] }{n+1}= \frac{1}{2\pi} \int_0^1 \int_0^{2\pi} \text{Tr} [\varphi(A(x,t))] \, dt \, dx
\end{equation*}
for any $\varphi \in C(D_{kA})$. In the non-normal case, the function $\varphi$ must be analytic in the closed disk $D_{kA} = \{ z : |z| \leq k \vvvert A \vvvert\}$. 
\end{theo}

Similarly, we can extend the perturbation result in Theorem \ref{altindex2} to the block matrix case. Once again, we choose the Frobenius norm for convenience. 

\begin{theo} \label{block2}
Let $\left\{ A_n \right\}$ with $A_n=\left[ A_{ij}^{(n)}\right]_{i,j=0}^n$ be a sequence of normal block matrices. Suppose the entries $A_{ij}^n$ are normal and satisfies the decay condition
\begin{equation} \label{decay2}
\mathcal{A}:=\sup_n  \left[ \sum_{k=0}^n  \max_{0\leq j\leq n-k} \left\| A_{j+k,j}^{(n)} \right\|_F + \sum_{k=1}^n  \max_{0\leq j\leq n-k} \left\| A_{j,j+k}^{(n)} \right\|_F \right] < \infty
\end{equation}
Suppose, further,  there exists a sequence $\{T_n(A)\}$ of generalized normal block KMS matrices for which \eqref{decay matrix} holds and 
\[  \left\| A_{ij}^{(n)} - \hat{A}_{j-i}\left(t_{i \wedge j}^{(n)} \right)\right\|_F = o \left( n^{-1/2} \right). \]
Let $\mathcal{M}=\max\{ \mathcal{A}, \vvvert A \vvvert \}$. If $A(x,t)$ is normal for all $x$ and $t$, then we have
\begin{equation*}
   \lim_{n \to \infty} \frac{ \text{Tr}\left[\varphi(A_n)\right] }{n+1} 
   =  \frac{1}{2\pi} \int_0^1 \int_{0}^{2\pi} \text{Tr} \left[ \varphi( A(x,t))  \right] \, dt \, dx
\end{equation*}
for every $\varphi \in C(D_{k\mathcal{M}})$. In the non-normal case, $\varphi$ must be analytic in the closed disk $D_{k\mathcal{M}}  =\{ z: |z| \leq k\mathcal{M}\}$. 
\end{theo}

\subsection{Finite rank perturbations}

The results above also hold for  finite rank perturbations.  In the Toeplitz case, this was first observed by Tyrtyshnikov \cite{ty96}. Let $\{A_n\}$ and $\{B_n\}$ be two sequences of normal  block-matrices. Suppose the spectrum of $A_n-B_n$ satisfies the bound
$$\|A_n -B_n \|_\infty =  \sup_{1 \leq k \leq k(n+1)} |\lambda_k(A_n-B_n)| =o( \sqrt{n}) ,$$ 
and  
$$ \text{rank}(A_n-B_n) \leq r$$
for some $r$ independent of $n$. Then, the sequences $\{A_n\}$ and $\{B_n\}$ have the same LSD.  Indeed, the Wielandt-Hoffman inequality  implies 
\begin{align*}
  &  \frac{1}{n+1} \sum_{j=1}^{k( n+1) } | \lambda_j(A_n) - \lambda_j(B_n)| \\ 
  & \qquad  \qquad \leq \frac{\sqrt{k} }{\sqrt{n+1}} \  \| A_n-B_n \|_F  = o \left( 1 \right)
\end{align*}   
since $\|A_n-B_n\|_F = o( r  \sqrt{n}) = o(\sqrt{n})$.

%


\begin{theo} \label{block3}
Let $\{A_n\}$ be as in Theorem~\ref{block2} and let $\{C_n\}$ be a sequence of block matrices with $\|C_n\|_\infty= o (\sqrt{n})$ and $ \mbox{rank}\,(C_n) \leq r$
for some positive constant $r$.  If $B_n = A_n + C_n$ is normal for all $n$, then
\[
  \lim_{n \to \infty} \frac{ \text{Tr}\left[\varphi(B_n)\right] }{n+1}    =  \frac{1}{2\pi}  \int_0^1\int_{0}^{2\pi} \text{Tr} [ \varphi( A(x,t))]   \, dt \, dx
\]
for any $\varphi \in C(D_{k\mathcal{M}})$.
\end{theo}

\subsection{Singular value distribution}
\label{singular}

The distribution of the singular values of sequences of KMS matrices can be obtained fairly easily once we have the preceding machinery.  A similar result for sequences of Toeplitz matrices goes back to  Avram \cite{av88} and Parter \cite{pa86} (see also \cite{bogr05}). Shao \cite{sh04} obtained a similar result using operator methods.

\begin{theo}
\label{SVDth}
  Let $T_{m,n}(a)$ be the $(m+1) \times (n+1)$ rectangular KMS matrix whose $(j,k)$ entry is given by 
  \[  \hat{a}_{j-i} \left(\frac{i+j}{2\max\{m,n\}+2}\right) 
  \]
  where the $\hat{a}_k$'s satisfy condition \eqref{decay}. Let 
  $$\sigma_1(T_{m,n}(a)) \geq \sigma_2(T_{m,n}(a)) \geq \cdots \geq \sigma_{m \wedge n}(T_{n,m}(a)) \geq 0 $$
  be the $m \wedge n = \min\{m+1,n+1\}$ singular values of $T_{m,n}(a)$.  Then, we have
  $$ \lim_{m,n \to \infty} \frac{1}{ m \wedge n } \sum_{k=1}^{m \wedge n} \varphi(\sigma_k(T_{m,n}(a))) = \frac{1}{2\pi} \int_0^1 \int_{0}^{2\pi}  \varphi(|a(x,t)|) \, dt \, dx$$
  for every $\varphi \in C([0,\vvvert a \vvvert])$.
\end{theo}

\begin{proof}
One can easily modify the trace formula in Theorem \ref{normal}  to obtain 
$$ \text{Tr}[ (T_{m,n}(a) T_{m,n}^{*}(a))^p ] =  \frac{m\wedge n}{2\pi} \int_0^1 \int_0^{2\pi} |a (x,t)|^p  \, dt \, dx +o(n)$$
for any $p \in \N$. Since the singular values of $T_{m,n}(a)$ are the eigenvalues of the Hermitian matrix $T_{m,n}(a)T_{m,n}^{*}(a)$, it suffices to apply Weierstrass Approximation Theorem  to obtain the desired result.
\end{proof}

Theorems \ref{block2} and \ref{block3} on sequences of block matrices can also be restated for the singular values.   

\begin{theo}
(i) Let $\left\{ A_{m,n} \right\}$  be a sequence of block matrices whose entries $A_{ij}^{(n)}$ are $k \times k$ matrices that satisfy the condition \eqref{decay2}. Suppose there exists a sequence $\{T_{m,n}(A)\}$ of generalized block KMS matrices for which \eqref{decay matrix} holds and 
\[  \left\| A_{ij}^{(n)} - \hat{A}_{j-i}\left(\xi_{i\wedge j}^{(n)} \right)\right\|_F = o \left( n^{-1/2} \right). \]
Then, we have
$$ \lim_{m,n \to \infty} \frac{1}{ m \wedge n } \sum_{j=1}^{k(m \wedge n)} \varphi(\sigma_j(A_{m,n})) = \frac{1}{2\pi} \int_0^1 \int_{0}^{2\pi} \sum_{j=1}^k  \varphi( \sigma_j(A(x,t))) \, dt \, dx$$
for every $\varphi \in C([0,k \mathcal{M}])$.

\smallskip

(ii) Let $A_{m,n}$ be as above and let $\{C_{m,n}\}$ be a sequence of block matrices of bounded rank satisfying
\[ \sigma_1(C_{m,n}) = \Oo(1) \qquad \text{and} \qquad  \mbox{rank}\,(C_{m,n}) \leq r \]
for some constant $r$.  Let $B_{m,n} = A_{m,n}+C_{m,n}$.  Then
  $$ \lim_{m,n \to \infty} \frac{1}{ m \wedge n } \sum_{j=1}^{k(m \wedge n)} \varphi(\sigma_j(B_{m,n})) = \frac{1}{2\pi} \int_0^1 \int_{0}^{2\pi}  \sum_{j=1}^k  \varphi( \sigma_j(A(x,t))) \, dt \, dx$$
for every $\varphi \in C([0,k \mathcal{M}])$.
\end{theo}

\subsection{Eigenvalue clustering for non-normal matrices}
\label{sec.clustering}

When $T_n(a)$ is not normal, \eqref{analyticeq1} holds, but only for analytic $\varphi$.  Since continuous functions cannot be approximated by analytic functions in the complex plane, this does not give us the LSD.  Indeed, it is easy to construct matrices where \eqref{analyticeq1} fails for continuous $\varphi$.  For example, consider the Toeplitz matrix with symbol $a(t) = e^{it}+e^{-2it}$, i.e. the matrix $[ \delta_{j+1,j} + \delta_{j,j+2}]$.  The eigenvalues of $T_n(a)$ all lie on the star $\{ r\omega^j : \omega = e^{2\pi i/3}, r\geq 0,  j=0,1,2\}$ in the complex plane \cite{mc09}.  The spectrum remains bounded away from the range of the symbol (see figure~\ref{toeplitz_ex}).  Thus, we can find a continuous $\varphi$ that is zero on the spectrum of $T_n(a)$ and positive on a portion of the range of $a$ with positive measure, which would contradict \eqref{analyticeq1} if it held for continuous $\varphi$.

\begin{figure}[h]
\begin{center}
\includegraphics*[width=2in]{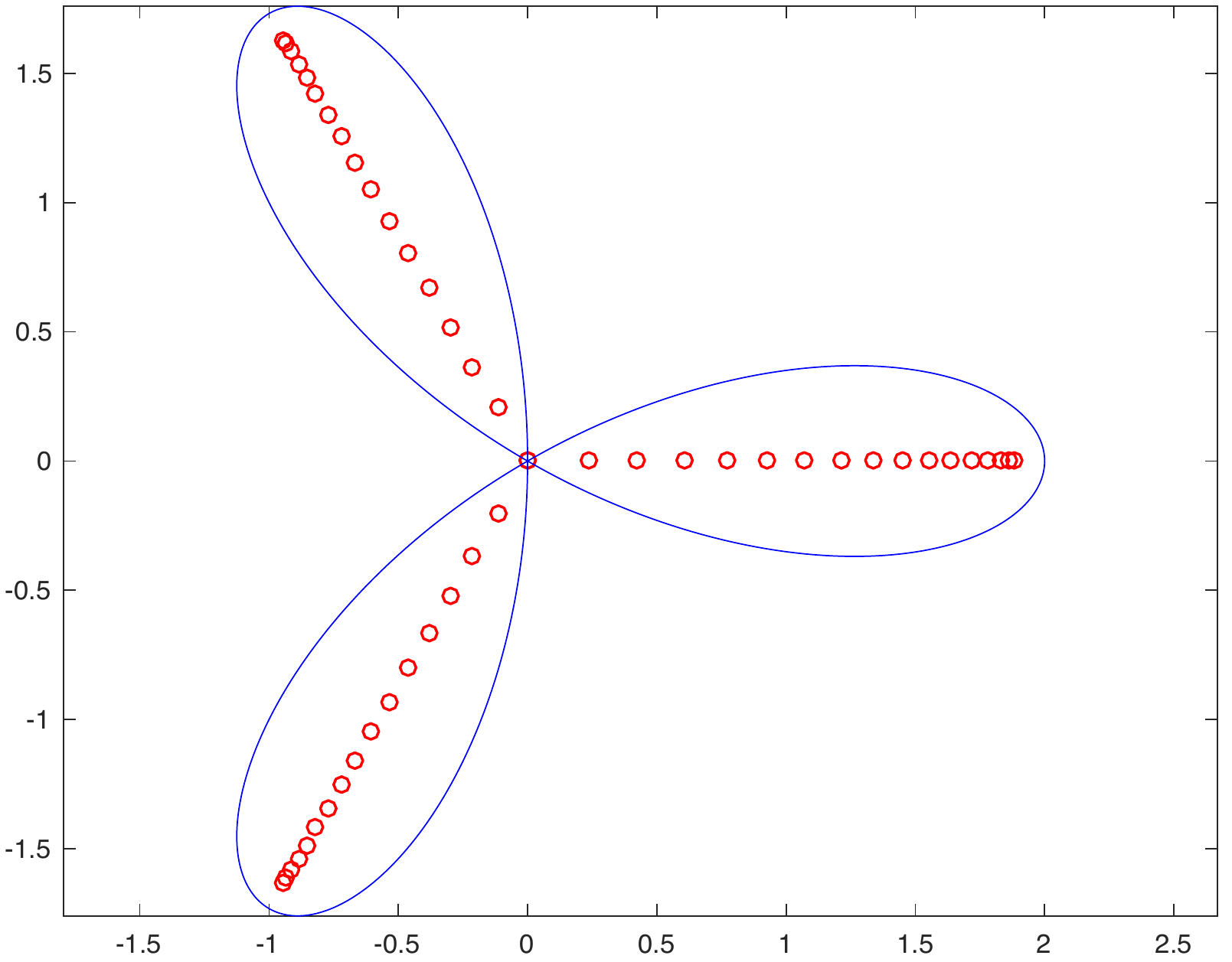}
\end{center}
\caption{Eigenvalues of $T_{50}(a)$ (circles) for $a(t) = e^{it}+e^{-2it}$, and  curve $a(t)$ for $t\in [0,2\pi]$.}
\label{toeplitz_ex}
\end{figure}

While we cannot deduce the LSD for non-normal KMS matrices, the symbol does give us some information about how the eigenvalues accumulate in the long run.  In this subsection, we will
extend a  result of Tilli \cite{ti99} about the clustering properties of the eigenvalues of Toeplitz matrices to KMS matrices. That is, we will  find a set in the complex plane such any of its neighborhoods contains all of the eigenvalues of $T_{n}(a)$, except at most $o(n)$ of them. We begin by defining the \textit{essential range} of the symbol $a(x,t)$, denoted $\mathcal{R}(a)$, by :

\begin{equation*}
\mbox{$\mathcal{R}(a):=\left\lbrace z \in \mathbb{C}: \mbox{meas} \left\lbrace a^{-1}(D_r(z)) \right\rbrace>0,\forall r>0\right\rbrace$},
\end{equation*}

\noindent 
That is, $\mathcal{R}(a)$ is the set of those points $z \in \mathbb{C}$ with the property that the Lebesgue measure of the inverse image of any open set containing $z$ is positive. $D_r(z)$ denotes an open disk in the complex plane with radius $r$ centered at $z$.

As before, we consider symbols that satisfy the decay condition \eqref{decay} and such that $a_t$ is Riemann integrable on $[0,1]$ and $a_x \in L^\infty( [0,2\pi])$. In such case,  $\mathcal{R}(a)$ is a compact set; hence its complement has just one unbounded connected component, and we can write

\begin{equation}
\mbox{$\mathbb{C} \backslash \mathcal{R}(a)=: U_{0}$  $\cup$ $\bigcup \limits_{j=1}^{\infty} U_{j}, \hspace{.25in} U_{i} \cap U_{j}=\emptyset$ if $ i\neq j  $},
\label{UNION}
\end{equation}
where each $U_{j}$, $j\geq 1$, is a connected bounded open set, and $U_{0}$ is an unbounded connected open set. Using (\ref{UNION}) we define the \textit{extended range} of the symbol $a$ as

$$\mathcal{ER}(a):=\mathbb{C} \backslash U_{0}
$$ 
Hence, the \textit{extended range} $\mathcal{ER}(a)$ is the union of the range of $a$ and all the bounded components of its complement. We can now state the main result of this section, which was proven by Tilli \cite{ti99} for Toeplitz matrices.  The proof that follows is only a minor modification of Tilli's.

\bth
\label{ClusterTheorem}
Let $a(x,t)$ be as above.  
Then, the extended range $\mathcal{ER}(a)$ is a cluster of the eigenvalues of $T_{n}(a)$.  That is, for any open set $V$ containing $\mathcal{ER}(a)$ there holds 
\begin{equation}
\lim \limits_{n\rightarrow \infty} \dfrac{\mathcal{N}(V,n)}{n+1}=1 ,
\label{Cluster}
\end{equation}
where $\mathcal{N}(V,n)$ is the number of  eigenvalues of $T_{n}(a)$ that lie inside $V$.  In other words,  any $\epsilon$-neighborhood of $\mathcal{ER}(a)$ contains all of the eigenvalues of $T_{n}$ except at most $o(n)$ of them.
\eth

\begin{proof}
Choose some $z \notin \mathcal{ER}(a)$. Since $\mathcal{ER}(a)$ is closed, there exists some small open disk $D$ centered at $z$ such that $\overline{D}$ $\cap$ $\mathcal{ER}(a)=\emptyset$. Let $K= \mathcal{ER}(a) \cup \overline{D}$ and define $F$ on $K$ as

$$F(\xi)=\left\{
               \begin{array}{ll}
               1 & $if $ \xi \in \overline{D}\\
               0 & $if $ \xi \in \mathcal{ER}(a).
               \end{array}
               \right.
$$

\noindent Since $\mathbb{C}\backslash K$ is connected, by the Mergelyan theorem we can uniformly approximate $F$ with a polynomial $P$, so that for any given $\varepsilon>0$,
$$|P(\xi)-F(\xi)| \leq \epsilon \qquad \qquad (\xi \in K).$$
Let $\mathcal{N}(D,n)$ be the number of eigenvalues of $T_{n}(a)$ inside $D$ and let $\chi_{D}$ be the characteristic function of $D$. Since $1-\epsilon \leq |P(\lambda)|$ whenever $\lambda \in D$, we have
\begin{align*}
(1-\epsilon) \ \mathcal{N}(D,n) &\leq  \sum \limits_{j=1}^{n+1} \chi_{D} (\lambda_{j} (T_{n}(a))) \, |P(\lambda_{j}(T_{n}(a)))|
\\
 &\leq   \left(\sum \limits_{j=1}^{n+1} \chi_{D} (\lambda_{j} (T_{n}(a))) \right)^{1/2} \  \left( \sum \limits_{j=1}^{n} |P(\lambda_{j}(T_{n}(a)))|^{2} \right)^{1/2}
 \\
& =   \mathcal{N}(D,n)^{1/2} \left( \sum \limits_{j=1}^{n+1} |P(\lambda_{j}(T_{n}(a)))|^{2} \right)^{1/2}.
\end{align*}
Since $P(\lambda_{j}(T_{n}(a)))=\lambda_{j}(P(T_{n}(a)))$, we can square both sides to obtain 
$$(1-\epsilon)^{2} \ \mathcal{N}(D,n) \leq   \sum\limits_{j=1}^{n+1} |P(\lambda_{j}(T_{n}(a)))|^{2} \leq  \ ||P(T_{n}(a))||_{F}^{2}.
$$
As noted in the previous section, one can modify the proof of Theorem \ref{normal} to obtain
\begin{align}
||P(T_{n}(a))||_{F}^{2} & =  \text{Tr}[ P(T_n(a)) \, (P(T_n(a)))^* ] \nonumber\\
 & =   \frac{n+1}{2\pi} \int_{0}^{1} \int_{0}^{2\pi} |P(a(x,t))|^{2} dt\, dx +o(n). \label{clust estimate}
\end{align}
The assumption of normality is not needed here since we only consider polynomials. Similarly, condition \eqref{decay} yields that $T_n(a)$ does not have to be banded for \eqref{clust estimate}. Consequently, we deduce 
$$\limsup_{n\rightarrow\infty} \frac{ (1-\epsilon)^{2} \, \mathcal{N}(D,n)}{n+1} \leq \frac{1}{2\pi} \int_{0}^{1} \int_{0}^{2\pi} |P(a(x,t))|^{2} \,dt \, dx \leq  \epsilon^{2}
$$ 
The last inequality holds since $|P(\xi)|\leq \epsilon$ whenever $\xi \in \mathcal{ER}(a)$.  From the arbitrariness of $\epsilon$, the last inequality implies $\mathcal{N}(D,n)=o(n)$. 

Now consider an arbitrary open set $V\supset \mathcal{ER}(a)$.  By Gershgorin's Theorem, the spectrum of $T_n(a)$ lies in  the union of some disks, say $G$. Let $C:=G \cap (\mathbb{C}\backslash V)$. If $C$ is empty, then there is nothing to prove. If $C \neq \emptyset$, then every $z \in C$ lies within some open disk $D(z)$ centered at $z$ for which
$$ \mathcal{N}(D(z),n)=o(n).$$
Thus, $C$ can be covered by a family of open disks $\left\lbrace D(z) \right\rbrace_{z\in C}$, each of which contains at most $o(n)$ eigenvalues. $C$ being a compact set, it can be cover by a finite sub-covering of those disks; hence $C$ itself must contain at most $o(n)$ eigenvalues of $T_{n}(a)$. Since $C \cup V$ contains all the spectrum of $T_{n}(a)$, equation \eqref{Cluster} follows and the proof is complete.
\end{proof}

To illustrate the above theorem, and also its limitations, consider the matrices $B_n$ whose entries on the $+1$ and $-2$ diagonals are given by
\[ b_{j,j-1} = 1 - \frac{\mu_{j-2}}{\mu_n}, \quad b_{j,j+2} = \frac{j(j+1)}{\mu_n}
\]
where
\[ \mu_n = n(n-1 + 6 \rho)
\]
The eigenvalue problem for $B_n$ arises when one seeks polynomial solutions to the generalized Lam\'{e} equation.  See \cite{mcboag09} for an interpretation of the eigenvalues of $B_n$ in terms of 
charges in a logarithmic potential.  It is easy to see that the $B_n$ are asymptotic to the KMS matrix with symbol
\[ a(x,t) = \left(1-x^2\right) e^{it} + x^2 e^{-2it}
\]
Theorem~\ref{ClusterTheorem} thus implies that the eigenvalues will cluster in the extended range of the symbol, which in this case is just the unit disk.  (Note that in this case the range and the extended range are the same.)  This is true, but it doesn't give us much detailed information about the spectrum.  In fact, the eigenvalues of $B_n$ all lie on the star $\{ r\omega^j : \omega = e^{2\pi i /3}, r\geq 0,  j=0,2,2\}$ in the complex plane \cite{mc09}.  One thing Theorem~\ref{ClusterTheorem} does give us is a bound on the magnitude of the eigenvalues (at least most of them).  This is illustrated in figure~\ref{lame_ex}.

\begin{figure}[h]
\begin{center}
\includegraphics*[width=2in]{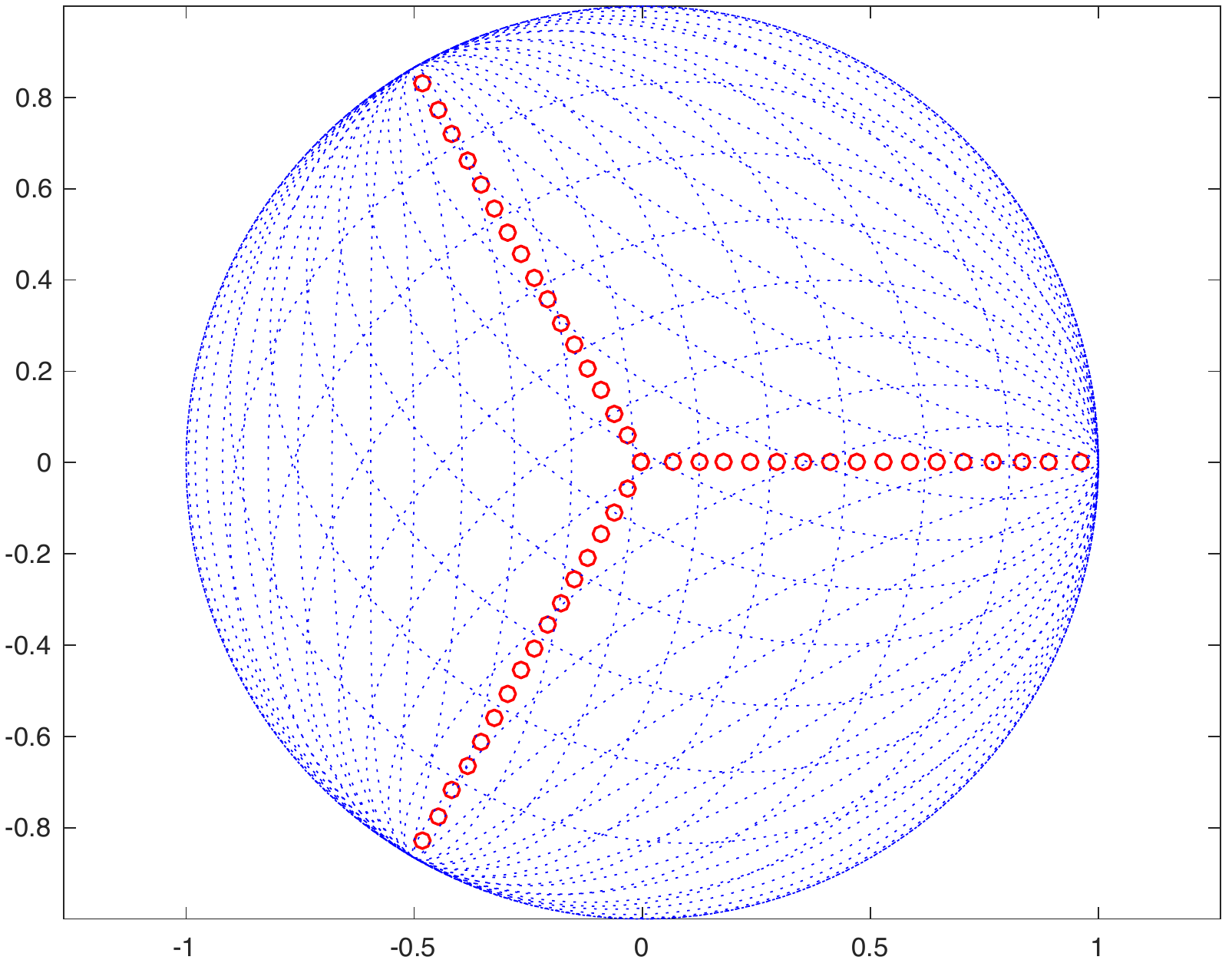}
\end{center}
\caption{Eigenvalues of $B_n$ (circles) for $n=50$ and $\rho = 1$, and curves $a(s,t)$ for $x=0,.05,\dots, 1$, $t\in [0,2\pi]$.  Compare with figure~\ref{toeplitz_ex}.}
\label{lame_ex}
\end{figure}

As another example, where the extended range is more informative, consider the symbol
\beq
 a(x,t) = \left(\frac{1}{2} + \sqrt{x - \frac{1}{2}}\right) e^{-2it} + e^{-it} + e^{it}
\label{clustersymbol}
\eeq
Eigenvalues and curves $a(x,t)$ for fixed values of $x$ are shown in figure~\ref{cluster_ex}.  We see that not all eigenvalues lie in the range of $a$, but they all lie in the extended range ${\cal ER}(a)$, which is roughly the interior of the envelope of the blue curves.

\begin{figure}[h]
\begin{center}
\includegraphics*[height=1.5in]{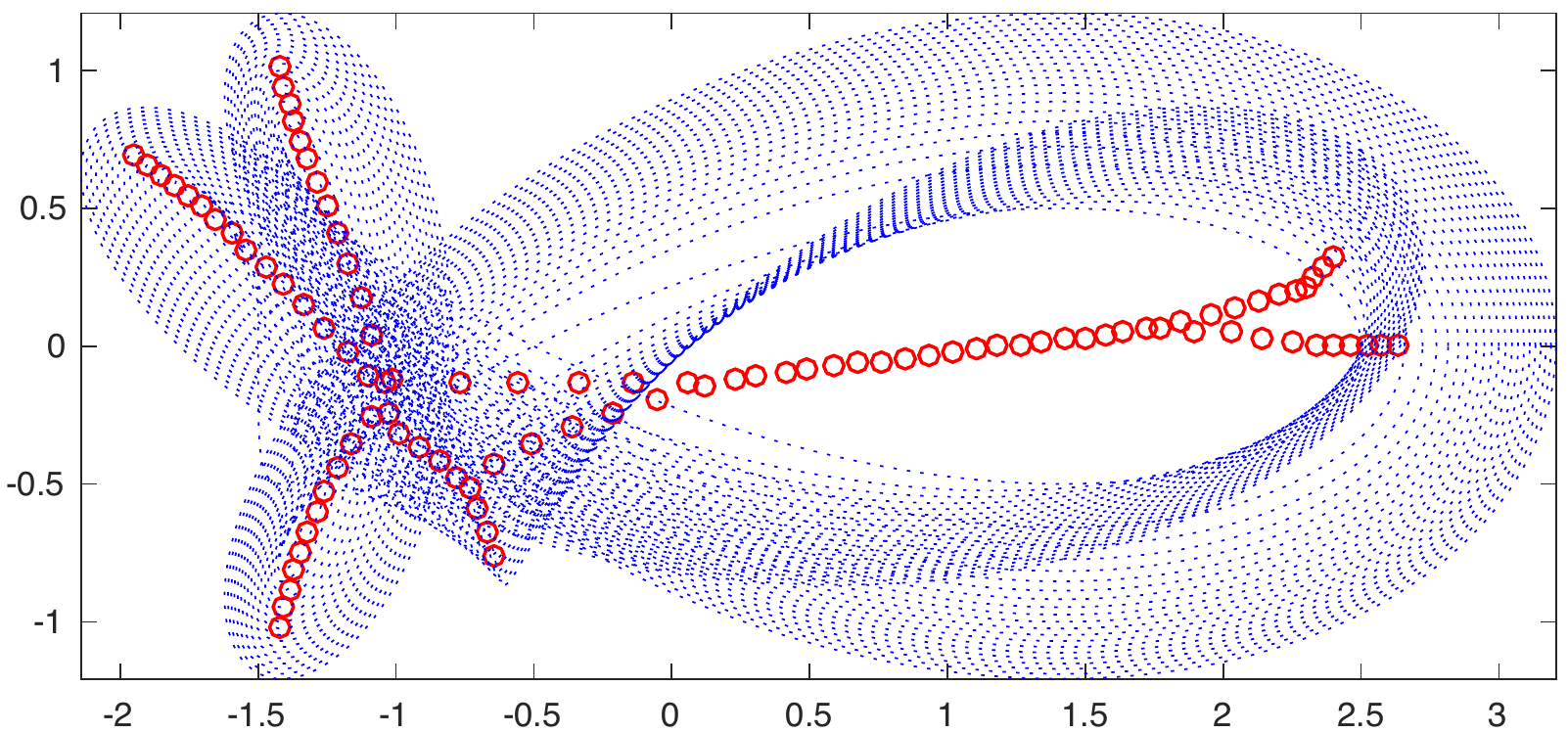}
\end{center}
\caption{Eigenvalues of $T_{100}(a)$ (circles) for $a$ as in \eqref{clustersymbol}, and curves $a(x,t)$ for $x=0,.025,\dots, 1$, $t\in [0,2\pi]$.}
\label{cluster_ex}
\end{figure}

%

%
%
\section{Strong Limit Theorems}
\label{sec.strong}

Now we turn to the Strong Limit Theorem, which can be viewed as a first order correction to the First Limit Theorem.
In order to prove a Strong Theorem for KMS matrices based on the moments method,  we need to obtain a more precise form of the error term in the proof of the First Limit Theorem. One way is to impose faster decay on the Fourier coefficients of the symbol $a$, or equivalently some smoothness on $a$. 

\subsection{Mejlbo and Schmidt's result and its generalizations}

We start by proving the following lemma whose proof follows essentially the same lines as the one found in \cite{mesc62}. We make some minor adjustments allowing us to relax the regularity conditions on the symbol.

\begin{lem}  \label{main lemma}
Let $a$ be a complex-valued symbol on $[0,1] \times [0,2\pi]$ such that  
$$a_x \in C^{1+\alpha}([0,2\pi]) \qquad \text{ and } \qquad a_t \in C^{1+\lceil 1/\alpha \rceil } ([0,1])$$
for some $0<\alpha \leq 1$. Moreover, suppose the following two conditions hold:
\begin{equation} \label{decay 2}
  \sum_{|k| \in \Z} |k|^{ 1+1/\alpha} \| \hat{a}_k \|_\infty < \infty \qquad \text{ and } \qquad   \sum_{|k| \in \Z}  \| \partial_x{\hat{a}}_k \|_{\varepsilon} < \infty.
\end{equation}
Then, for any $p \in \N$, we have 
\begin{align*}
 \text{Tr}[T_n^p( a) ] & =  \frac{n+1}{2\pi} \int_0^1 \int_0^{2\pi} a^p(x,t) \, dt \, dx  +  \frac{1}{4\pi} \int_0^{2\pi} a^p(0,t) \, dt - \frac{1}{4\pi} \int_0^{2\pi} a^p(1,t) \, dt \\
   & \quad + \sum_{|h|=0} N(h) \, \left[ \prod_{j=1}^p \hat{a}_{h_j}(0) + \prod_{j=1}^p \hat{a}_{h_j}(1) \right] + p \,  \vvvert a \vvvert^{p-1} \, o(1)
 \end{align*}
 where $N(h) =  \max\{0,h_1,h_1+h_2,....,h_1+\cdots + h_{p-1} \} $. 
 \end{lem}  

\begin{proof}
 Pick $0<\delta < \alpha/(1+\alpha)$ and let  $b(x,t)$ be the truncated symbol defined by
 $$ b(x,t) = \sum_{|k| \leq (n+1)^\delta} \hat{a}_k(x) \ e^{ikt}.$$
In order to simplify several of the expressions below, we introduce some useful notations. For every multi-index $h=(h_1,...,h_p)$ such that $|h|=0$, we denote by  
$$H=\{ h_1, h_1+h_2,...,h_1+ \cdots + h_{p-1} \}$$ 
and by $m(H)=\min H$ and $M(H)=\max H$.  We also denote by $S_{i,\delta}$ the set 
$$ S_{i,\delta} = \{ h : \, \max_j |h_j| \leq  (n+1)^\delta, \ -i  \leq m(H) \leq M(H) \leq n-i\}.$$
Finally, we introduce the functions 
$$  f_{h,i} (t) = \prod_{k=1}^p  \hat{a}_{h_k} \left( \frac{i}{n+1} + t \, \frac{2h_1+\cdots + 2h_{k-1} +h_k}{2n+2} \right).$$
With $\vvvert a \vvvert $ be as in \eqref{decay}, the trace formula \eqref{trace 2} together with the bounds \eqref{decay 2} yield
\begin{align*}
 \left|  \text{Tr}[T^p_n(a)] - \text{Tr}[T^p_n(b) ]\right| & =  \left|  \sum_{i=0}^n \sum_{h \in S_{i,1}}  f_{h,i}(1) -  \sum_{i=0}^n \sum_{h \in S_{i,\delta}}  f_{h,i}(1) \right| \\
   & \leq  \sum_{i=0}^n \sum_{h \in S_{i,1} - S_{i\delta}} |f_{h,i}(1)| \\
  & \leq  p \vvvert a \vvvert^{p-1} (n+1)   \sum_{|k|>(n+1)^{\delta}}  \| \hat{a}_k\|_\infty\\
  & \leq  p \vvvert a \vvvert^{p-1} \sum_{|k|> (n+1)^{\delta}} |k|^{1+1/\alpha} \| \hat{a}_k \|_\infty.
\end{align*}  
Thus, condition  \eqref{decay 2} implies for $n$ large enough that 
\begin{equation*}
 \left|  \text{Tr}[T^p_n(a)] - \text{Tr}[T^p_n(b) \right|  = p \vvvert a \vvvert^{p-1} \, o(1).
\end{equation*}
If we denote by $T^{tr}_n(b)$ the transpose matrix -- not the conjugate-transpose -- of $T_n(b)$, then another application of \eqref{trace 2} gives us
\begin{align*}
     \text{Tr}[(T^{tr}_n (b))^p ] & =   \prod_{k=1}^p  \hat{a}_{h_k} \left( \frac{i}{n+1} -  \, \frac{2h_1+\cdots + 2h_{k-1} +h_k}{2n+2} \right) \\
       & =   \sum_{i=0}^n \sum_{h \in S_{i,\delta} } f_{h,i}(-1).
\end{align*}
Consequently, we obtain
\begin{align*}
 \text{Tr}[T_{n}^p(b)] - \sum_{i=1}^n \sum_{h \in S_{i,\delta}}  f_{h,i}(0) & =  \frac{1}{2} \, \left( \text{Tr}[T_n^p(a_{(n+1)^\delta})] + \text{Tr}[((T_n^{tr})^p)(a_{(n+1)^\delta})] \right)\\
   & \quad -  \sum_{i=0}^n \sum_{|h|=0}  f_{h,i}(0) \\
  & =   \frac{1}{2} \sum_{i=0}^n \sum_{ h \in S_{i,\delta}} \left(  f_{h,i}(-1) + f_{h,i}(1) - 2 f_{h,i}(0) \right)
  \end{align*}    
By applying the Mean Value Theorem to $f_{h,i}(-1)-f_{h,i}(0)$ and $f_{h_i}(1)-f_{h,i}(0)$ and using the H\"older continuity of $\partial_x a$, we can bound the previous expression by
\begin{align*}
   \left|  \text{Tr}[T_{n}^p(b)] -  \sum_{i=0}^n \sum_{h \in S_{i,\delta}}  f_{h,i}(0)  \right|  & \leq  \frac{(n+1) \, p \, \vvvert a \vvvert^{p-1}}{2} \, \left( \frac{ (n+1)^\delta}{n+1} \right)^{1+1/\alpha} 
   \\
   &= p \,  \vvvert a \vvvert^{p-1} o(1)
 \end{align*}
from our choice of $\delta$. By a similar argument as above, we also obtain
\begin{equation*}
  \left|  \sum_{i=0}^n \sum_{h \in S_{i,1}}  f_{h,i}(0) -  \sum_{i=0}^n \sum_{h \in S_{i,\delta}}  f_{h,i}(0)  \right| = p \vvvert a \vvvert^{p-1} \, o(1).
\end{equation*}   
Therefore, it suffices to prove the result for the sum $ \sum_{i=0}^n \sum_{h \in S_{i,\delta}}  f_{h,i}(0)$. We have
 \begin{equation} \label{MS1}
  \sum_{i=0}^n \sum_{h \in S_{i,1}}  f_{h,i}(0) =  \sum_{i=0}^n \sum_{|h|=0}  f_{h,i}(0) - \sum_{i=0}^n \sum_{h \in R_{i,1}} f_{h,i}(0) -  \sum_{i=0}^n \sum_{h \in R_{i,2}} f_{h,i}(0) 
\end{equation}    
where $R_{i,1}=\{ h: |h|=0, \ m(H)<-i, \ M(H) \leq n-i\} $ and $R_{i,2} =\{ h : \  |h|=0, \ M(H)>n-i\}$. 
Arguing as in the proof of the First Theorem and applying the Euler-Maclaurin formula, we see 
\begin{align} \label{MS2}
    \sum_{i=0}^n \sum_{h \in S_{i,1}}  f_{h,i}(0)  & =  \sum_{i=0}^n \int_0^{2\pi} a^p(i/n,t) \, dt \nonumber\\
     & =  (n+1) \int_0^1 \int_0^{2\pi}  a^p(x,t) \, dt \, dx + \frac{1}{4\pi} \int_0^{2\pi} a^p(0,t) \, dt \nonumber \\
     & \quad - \frac{1}{4\pi} \int_0^{2\pi} a^{p}(1,t) \, dt +  p \vvvert a \vvvert^{p-1}  o(1). 
\end{align} 
On the other hand, 
\[
  \sum_{i=0}^{n} \sum_{h \in R_{i,1}} |f_{h,i}(0)|   \leq  \sum_{|h|=0} \left[ \sum_{k=1}^p |h_k| \right] \prod_{k=1}^p \|\hat{a}_k\|_\infty  \leq p \left[ \sum_{k \in \Z} |k| \|\hat{a}_k\|_\infty \right]  \vvvert a \vvvert^{p-1},
\]
so we can apply the Dominated Convergence Theorem to conclude 
\begin{align} 
\label{MS3}
   \sum_{i=0}^n \sum_{h \in R_{i,1}}  f_{h,i}(0) &  =  \sum_{i=0}^n \sum_{h \in R_{i,1}} \prod_{k=1}^p \hat{a}_{h_k}(0) + p \, \vvvert a \vvvert^{p-1} o(1) \nonumber\\
    &  =  \sum_{|h|=0} N(h) \prod_{k=1}^p \hat{a}_{h_k}(0) + p \, \vvvert a \vvvert^{p-1} o(1).
\end{align}   
In the same way, we prove 
\begin{align} 
\label{MS4}
  \sum_{i=0}^n \sum_{h \in R_{i,2}}  f_{h,i}(0) & =  \sum_{i=0}^n \sum_{ |h|=0 \atop M(H)>i} \prod_{k=1}^p \hat{a}_{h_k}\left( \frac{n-i}{n+1} \right) \nonumber \\
  & =   \sum_{|h|=0} N(h) \prod_{k=1}^p \hat{a}_{h_k}(1) + p \, \vvvert a \vvvert^{p-1} o(1).
\end{align}   
The result follows by combining \eqref{MS1}, \eqref{MS2}, \eqref{MS3} and \eqref{MS4}.
\end{proof}

 The Strong Theorem for KMS matrices below is a modest improvement of the one's obtained by Mejlbo and Schmidt. We use the same notation as in the introduction.

\begin{theo}  
Let $a$ be as in Lemma \ref{main lemma}. If  $\text{Re} \, (a) > 0$, then 
\beq \label{Strong KMS eq}
 \lim_{n \to \infty}  \frac{\det T_n(a)}{(G(a))^n}   = \exp  \,\frac{1}{2} \{ e(a;0)-e(a;1) +E(a;0)+E(a;1) \}.
 \eeq
\end{theo}

\begin{proof}
Let $b=a-1$ and let $\vvvert a \vvvert $ and $\vvvert b \vvvert$ be the bounds \eqref{decay} associated to $a$ and $b$. The left hand side of \eqref{Strong KMS eq} is invariant under scaling of the symbol, therefore we may assume without loss of generality that $\vvvert a \vvvert<1$ and $\|a\|_\infty < 1$ with similar bounds for $b$.  In particular, the eigenvalues of $T_n(b)$ satisfy $|\lambda_k(T_n(b))|<1$. Thus, 
\begin{align*}
    \log (\det T_n(a) ) & =  \sum_{k=0}^n \log \lambda_k(T_n(a))\\
      & =  \sum_{k=0}^n \sum_{p=1}^\infty \frac{1}{p} \ (\lambda_k(T_n(a))-1)^p\\
      & =  \sum_{k=0}^n \sum_{p=1}^\infty \frac{1}{p} \ (\lambda_k(T_n(b)))^p.
\end{align*}    
Since $\vvvert b \vvvert<1$, the double sum is absolutely convergent, and hence we can switch the order of summation.  By previous  lemma applied to the symbol $b$, it then follows
\begin{align*}
   \log (\det T_n(a) ) 
    & =   \sum_{p=1}^\infty \frac{1}{p} \left\{ \frac{n+1}{2\pi} \int_0^1 \int_0^{2\pi} b^p(x,t) \, dt \, dx \right.\\
    & \quad + \frac{1}{4\pi} \int_0^{2\pi} b^p(0,t) \, dt - \frac{1}{4\pi} \int_0^{2\pi} b^p(1,t) \, dt \\
    & \quad + \left. \sum_{|h|=0} N(h) \, \left[ \prod_{j=1}^p \hat{b}_{h_j}(0) + \prod_{j=1}^p \hat{b}_{h_j}(1) \right] \right\}+o(1).
\end{align*}
That is, 
\begin{align*}
   \log (\det T_n(a) ) 
     & =    \frac{n+1}{2\pi} \int_0^1 \int_0^{2\pi}  \log a(x,t) \, dt \, dx \\
    & \quad +\frac{1}{4\pi} \int_0^{2\pi} \log a(0,t) \, dt - \frac{1}{4\pi} \int_0^{2\pi} \log a(1,t) \, dt \\
    & \quad + \sum_{p=1}^\infty \frac{1}{p}    \sum_{|h|=0} N(h) \, \left[ \prod_{j=1}^p \hat{b}_{h_j}(0) + \prod_{j=1}^p \hat{b}_{h_j}(1) \right]  +o(1).
\end{align*}   
Note that for $x=0$ or $x=1$, the matrices $T_n(b)$ are Toeplitz. By the Strong Szeg\H{o}'s Theorem for Toeplitz matrices, we deduce 
\begin{eqnarray*}
    \lefteqn{ \sum_{p=1}^\infty \frac{1}{p}    \sum_{|h|=0} N(h) \, \left[ \prod_{j=1}^p \hat{b}_{h_j}(0) + \prod_{j=1}^p \hat{b}_{h_j}(1) \right] }\\
     & & = \frac{1}{2} \sum_{k \in Z}  \left[ k\,  \widehat{(\log a)}_k(0) \, \widehat{(\log a)}_{-k}(0)+ \sum_{k \in \Z} k \, \widehat{(\log a)}_k(1) \, \widehat{(\log a)}_{-k}(1) \right] +o(1).
 \end{eqnarray*}     
The desired conclusion is then an immediate consequence of last two estimates.
\end{proof}

It is well-know that the Fourier coefficients decay faster when more regularity of the symbol is assumed. To this extent, recall the estimate
\begin{equation*} \label{fourier coeff}
   \| \hat{a}_k\|_\infty \leq \frac{\| a\|_{r+\alpha}}{(1+|k|)^{r+\alpha}}
\end{equation*}   
where $\| \cdot \|_{r+\alpha}$ denotes the H\"older norm with $r \in \N$ and $0 < \alpha \leq 1$.  For instance, both conditions in \eqref{decay 2} hold if we assume that $a_x \in C^{2+\alpha}([0,2\pi])$ and $a_t \in C^{2+\lceil 1/\alpha \rceil}([0,1])$.  

\begin{cor}
Let $a$ be a complex-valued symbol on $[0,1] \times [0,2\pi]$ with $\text{Re}\, (a) > 0$. Suppose $a_t \in C^{2+\alpha}([0,1])$ and $a_x \in C^{2+\lceil 1/\alpha \rceil } ([0,2\pi])$ 
for $0<\alpha \leq 1$. Then, $\{T_n(a)\}$ satisfies the conclusion of the Strong Limit Theorem above.
\end{cor}

Similarly, if $a$ is merely a trigonometric polynomial, or equivalently the matrices $T_n(a)$ have fixed band size, then $a_x$ is smooth on $[0,2\pi]$. Hence, we obtain the following consequence.

\begin{cor}
The Strong Theorem holds if $a$ is a trigonometric polynomial on $[0,1] \times [0,2\pi]$ with $\text{Re}(a)> 0$, and $a_t \in C^{1+\alpha}([0,1])$  for any $\alpha >0$. 
\label{STtrig}
\end{cor}

\subsection{Widom's result for KMS matrices}

We include here  an interesting extension of the Strong Limit Theorem  due to Widom \cite{wi76} in the Toeplitz case.  The Strong Limit Theorem \eqref{Strong KMS eq} is written using $\varphi = \log$.  However, this can be generalized.
 If $\{T_n(a) \}$ is a sequence of Toeplitz matrices whose symbol is absolutely continuous on $[0,2\pi]$, then Widom obtained a Strong Theorem for arbitrary analytic functions other than the logarithm. Namely,
\begin{eqnarray*}
   \lefteqn{ \lim_{n \to \infty}  \left[ \text{Tr}[ \varphi(T_n(a)) ] - \frac{n+1}{2\pi}  \int_0^{2\pi} \varphi(a(t)) \,  dt \right]  }\\
   & & = \frac{1}{8\pi^2} \sum_{k=1}^\infty  \int_0^{2\pi} \int_0^{2\pi} \sin(k(t-\tau)) \,  \frac{\varphi(a(t)) - \varphi(a(\tau))}{a(t)-a(\tau)} \, (  a'(t) -  a'(\tau)) \, dt \, d\tau.
\end{eqnarray*}
In the theorem below, we extend Widom's result to sequence of KMS matrices.

\begin{theo} Let $a$ be as in Lemma \ref{main lemma}. Then, for any $\varphi \in C^\omega(D_{M})$
\begin{align*}
    \lim_{n \to \infty} & \left[ \text{Tr}[ \varphi(T_n(a)) ] - \frac{n+1}{2\pi} \int_0^1 \int_0^{2\pi} \varphi(a(x,t)) \, dt \, dx \right]  \\
     &  =   \frac{1}{4\pi} \int_0^{2\pi} \varphi(a(0,t)) - \varphi(a(1,t)) \, dt \\
      & \quad + \frac{1}{4\pi^2} \sum_{k=1}^\infty  \int_0^{2\pi} \int_0^{2\pi} \, \left( \Phi(0,t,\tau)+ \Phi(1,t,\tau) \right) \, \sin (k(t-\tau))  \,  dt \, d\tau 
\end{align*}       
where $$\Phi(x,t,\tau) = \frac{\varphi(a(x,t)) - \varphi(a(x,\tau))}{a(x,t)-a(x,\tau)} \, ( \partial_t a(x,t) - \partial_t a(x,\tau)).$$
\end{theo}

\begin{proof} Write $\varphi(z) = \sum_{r=0}^\infty c_r z^r$ for $|z| <  \vvvert a \vvvert$. From Lemma \ref{main lemma}, we deduce that 
\begin{align*}
  \text{Tr} [\varphi (T_n(a) )] & -  \frac{n+1}{2\pi} \int_0^1 \int_0^{2\pi} \varphi(a(x,t)) \, dt \, dx \\
   &  = \frac{1}{4\pi} \int_0^{2\pi} \varphi (a(0,t)) \, dt - \frac{1}{4\pi} \int_0^{2\pi} \varphi( a(1,t)) \, dt \\
   &  \quad + \sum_{r=0}^\infty c_r    \sum_{|h|=0} N(h) \, \left[ \prod_{j=1}^r \hat{a}_{h_jd}(0) + \prod_{j=1}^r \hat{a}_{h_j}(1) \right]  +o(1).
\end{align*}   
Using the fact that $T_n(a)$ is Toeplitz for fixed $s$,  we obtain from Widom's reult
\begin{align*}
  \ \sum_{r=0}^\infty c_r  & \sum_{|h|=0} N(h) \, \left[ \prod_{j=1}^r \hat{a}_{h_jd}(0) + \prod_{j=1}^r \hat{a}_{h_j}(1) \right] \\
     &  =    \frac{1}{8\pi^2} \sum_{k=1}^\infty  \int_0^{2\pi} \int_0^{2\pi} \sin (k(t-\tau)) \, (\Phi(0,t,\tau)+\Phi(1,t,\tau)) \, dt \, d\tau  + o(1)
\end{align*}
from which the desired result follows.
\end{proof}

\subsection{Ehrhardt and Shao's result}

In a series of papers Shao \cite{sh98,sh03} and Ehrhardt  \cite{ehsh01} generalized the elegant operator method of Widom \cite{wi74, wi76} to
extend the results of Kac, Murdock and Szeg\H{o}, and Mejlbo and Schmidt to symbols with less regularity.  First, Shao \cite{sh98} proved a first limit theorem for symbols of bounded variation.  Later Ehrhardt and Shao \cite{ehsh01} proved a strong theorem for symbols with less regularity than Mejlbo and Schmidt required.  Shao \cite{sh03} then generalized some of these results to block matrices.

\smallskip

Ehrhardt and Shao define their matrices differently than Kac, Murdock and Szeg\H{o}.  In particular, given a function $a$ with Fourier series \eqref{FS1}, Ehrhardt and Shao define the  matrices
\newcommand{\opn}{\mbox{op}_n}
\beq
\opn a = \left[ \hat{a}_{j-i}\left(\frac{i}{n}\right)\right]_{i,j=0}^n
\label{opdef}
\eeq
These are not KMS matrices due to the indexing by $i/n$, as opposed to $(i+j)/(2n+2)$.  A peculiarity of this definition is that it is difficult to find a condition on the symbol $a$ so that $\opn a$ will be Hermitian.  Hence, it is a somewhat unnatural definition.

By Theorem~\ref{altindex}, the LSD of such matrices is the same as for KMS matrices (a result derived in \cite{sh98}).  However, Ehrhardt and Shao found the following for the determinant.  Suppose $a$ has winding number zero, $a$ is $C^{1+\alpha}$ in $x$ and $C^{\alpha}$ in $t$.  Then
\beq
\lim_{n\rightarrow\infty} \frac{\det (\opn a)}{G(a)^{n+1}} = \exp\frac{1}{2}\left\{e(a;0)+e(a;1) + E(a;0) + E(a;1) + {\cal F}(a)\right\}
\label{ES1}
\eeq
where $G, e$ and $E$ are as in (\ref{GeE1}-\ref{GeE3}), and
\[ {\cal F}(a) =  \int_0^1 \left(\sum_{k=-\infty}^{\infty} k \widehat{(\log a)}_k(x) \widehat{(\partial_x \log a)}_{-k}(x)\right) dx
\]
\begin{rem}
The paper of Ehrhardt and Shao \cite{ehsh01} contains a typo in the formula for $\lim_{n\rightarrow\infty} \det (\opn a)/G(a)^{n+1}$.  Their formula  (eqn (1.5) in \cite{ehsh01}) contains $ G(a)^{-1}$.  This would imply that the exponent of $G(a)$ in the denominator of \eqref{ES1} is $n$, instead of $n+1$, which is incorrect.
\end{rem}

The formula \eqref{ES1} of Ehrhardt and Shao appears to contradict the result \eqref{Strong KMS eq} of Mejlbo and Schmidt.  This difference was noted (but not explained) by Ehrhardt and Shao.  We will attempt to explain the difference here.  First there is the sign difference between $e(a;0)$ and $e(a;1)$.  However, this is only due to the fact that the KMS matrices index from $0$ to $n/(n+1)$ along the diagonals, whereas Ehrhardt and Shao index from $0$ to $1$.  This sensitivity was pointed out by Kac \cite{ka69}.  Changing the indexing in \eqref{opdef} from $i/n$ to $i/(n+1)$ would have the effect of changing the sign in \eqref{ES1} to $e(a;0)-e(a;1)$.  This difference arises from the Euler-Maclaurin approximation of the integral.

\smallskip

The more serious difference, though, is the presence of the ${\cal F}(a)$ term in \eqref{ES1}, which does not appear in \eqref{Strong KMS eq}.  Indeed, whereas Mejlbo and Schmidt's formula depends  on the symbol $a$ only at $x=0$ and $x=1$, the formula of Ehrhardt and Shao depends on $a$ for \emph{all} $x\in[0,1]$.  And, as Ehrhardt and Shao point out, $\exp\{{\cal F}(a)\}$ can take on any nonzero constant $c$ by choosing $a$ so that
\[ \log a(x,t) = 2e^{it}\log c + e^{-it} x
\]
Thus, the two formulas are incompatible.  

\smallskip

The formula \eqref{ES1} does not contradict \eqref{Strong KMS eq}  because $\opn a$ and $T_n(a)$ are different matrices.  While the LSD of the eigenvalues is the same for $\opn a$ and $T_n(a)$, the determinants behave differently in the limit.  Determinants are much more sensitive to small changes than are eigenvalue distributions.  The ${\cal F}(a)$ term in \eqref{ES1} is an artifact of the way Ehrhardt and Shao define their matrices.

The formula \eqref{ES1} comes from a generalization of the operator results of Widom.  Recall that
the standard Toeplitz and Hankel operators acting on $l^2$ are defined for symbols of one variable by
\[ T(a) = \left[\hat{a}_{j-i}\right]_{i,j=0}^{\infty} \quad \mbox{ and } \quad
H(a) = \left[\hat{a}_{i+j+1}\right]_{i,j=0}^{\infty}
\]
The product of Toeplitz operators is Toeplitz+Hankel, i.e.
\[ T(ab) - T(a)T(b) = H(a) H(\tilde{b})
\]
where $\tilde{b}(z) = b(z^{-1})$.
Widom \cite{wi74, wi76} derived a finite dimensional analogue:
\beq
T_n(ab) - T_n(a)T_n(b) = P_n H(a) H(\tilde{b}) P_n + Q_n H(\tilde{a}) H(b) Q_n
\label{w1}
\eeq
where $P_n$ and $Q_n$ are the projection and flip operators:
\begin{align*}
 P_n (f_0, f_1, \dots) &= (f_0, \dots, f_n, 0, \dots) 
 \\
Q_n (f_0, f_1, \dots) &= (f_n, f_{n-1}, \dots, f_0, 0, \dots)
\end{align*}
Widom then employed this formula in a beautiful argument using operator theory to
prove the Strong Szeg\H{o} Theorem \eqref{SS2}, showing that the error term can be written as
\[
 \exp\{E(a)\} = \det T(a)T(a^{-1})
\]

Ehrhardt and Shao generalize this result to matrices of the form \eqref{opdef}.  First, they generalize \eqref{w1} to 
\beq
\opn a(x)b(y) - (\opn a)(\opn \overline{b})^* = H_n(a) H_n(\tilde{b})^{tr} + J_n(a)J_n(\tilde{b})^{tr}
\label{ES3}
\eeq
where the matrix $\opn a(x)b(y)$ is the matrix whose $(p,q)-$entry is given by
\[ \frac{1}{2\pi} \int_0^{2\pi} a\left(\frac{p}{n},t\right) b\left(\frac{q}{n},t \right) e^{-i(p-q)t} dt
= \sum_{k=-\infty}^{\infty}\hat{a}_{p-k}\left(\frac{p}{n}\right) \hat{b}_{k-q}\left(\frac{q}{n}\right)
\]
and $J_n(a)$ and $H_n(a)$ are the half-infinite matrices given by
\[ J_n(a) = \left[ \hat{a}_{-1-n+j-i}\left(\frac{i}{n}\right) \right],
\
H_n(a) = \left[\hat{a}_{1+i+j}\left(\frac{i}{n}\right)\right],
 \ \left( \begin{array}{l}  0\leq i \leq n, \\[.05in] j=0,1,2,\dots\end{array}\right)
\]
Once the formula \eqref{ES3} is established they are able, with some difficulty, to carry through the operator argument to establish the generalized Strong Szeg\H{o} Theorem \eqref{ES1}.

\smallskip

It appears that this argument does not work for KMS matrices.  When one indexes by $(i+j)/(2n+2)$ along the diagonals, the formula \eqref{ES3} breaks down, and there does not seem to be a consistent way to define something analogous to $\opn a(x)b(y)$ to make it work.  Ehrhardt and Shao's way of defining their matrices works since one only indexes by the row.  This makes the definitions used in \eqref{ES3} consistent, and allows them to carry through the computation.  The term ${\cal F}(a)$ arises as a perhaps unfortunate side effect.

%
%
%
%

%
%
\medskip
\subsection {The discrete Schr\"{o}dinger operator }
\label{DSOsection}

In this section we report on results for the important special case of the discrete Schr\"{o}dinger operator
\beq 
T_n(f) = \begin{bmatrix} f(\frac{1}{n}) & -1 & 0 & \cdots & 0 \\
-1 & f(\frac{2}{n}) & -1 & \cdots & 0 \\
0 & -1 & f(\frac{3}{n}) & \cdots & 0 \\
 &  &  & \ddots
 \\
0 & 0 & 0 & \cdots & f(\frac{n}{n}) 
\end{bmatrix}
\label{DSOdef}
\eeq
In this subsection we use a slightly different notation from the rest of the paper, in order to be consistent with that of M. Kac, whose paper \cite{ka69} contains the results on which the results of this subsection are based on.  As the proofs, which are found in \cite{bomc16}, are somewhat lengthy and technical, they are omitted.  They are obtained by modifying Kac's proof in \cite{ka69} for $T_n(f)$ when $f$ is twice differentiable.

The first theorem  \eqref{First KMS} obviously holds for $T_n(f)$, as long as $f$ is Riemann integrable, and tells us that the eigenvalues accumulate like the values of $a(x,t) = f(x)-2\cos t$ sampled at regularly spaced points in $[0,1]\times [0,2\pi]$.  

The second theorem \eqref{Strong KMS} holds as long as $f\in C^{1+\alpha}([0,1])$ for some $\alpha > 0$.  Let
\[ D_n(f) = \det T_n(f)
\]
Then, in the notation of this section,
\[ \lim_{n\rightarrow\infty} \frac{ D_n(f)}{G(f)^n} = E(f)
\]
where
 \[G(f) = \exp \left\{ \int_0^1 \log\left(\frac{f(x) + \sqrt{f^2(x)-4}}{2}\right) dx \right\}
\]
is the geometric mean of $f(x)-2\cos t$, and $E(f)$ is given in terms of the series of Fourier coefficients of $f(x)-2\cos t$.  (The terms in \eqref{Strong KMS} have to be modified slightly for the different convention used by Kac.)
  
 \smallskip

In 1969 Kac \cite{ka69} derived a beautiful and simple formula for $E(f)$ for this case.  The following theorem, dealing with  families of matrices with any shift in the indexing, is obtained by modifying Kac's original proof.

\begin{theo}
\label{scaletheorem}
Let $f$ be twice differentiable on some open interval $I$ containing $[0,1]$.  Suppose $f$ has  a bounded second derivative and satisfies $f>2$.  Let $\ve\in \mathbb{R}$ and define 
the matrices
\[
T_{n}(f;\ve ) =  \begin{bmatrix} f(\frac{\ve }{n}) & -1 & 0 & \cdots & 0 \\
-1 & f(\frac{1+\ve }{n}) & -1 & \cdots & 0 \\
0 & -1 & f(\frac{2+\ve }{n}) & \cdots & 0 \\
 &  &  & \ddots
 \\
0 & 0 & 0 & \cdots & f(\frac{n-1+\ve }{n}) 
\end{bmatrix}
\]
Then 
\beq 
 \lim_{n \to \infty}  \frac{\det {T}_{n}(f;\ve)}{G(f)^n}   =\frac{\dsp \left( f(0)+\sqrt{f^2(0 ) - 4}\right)^{1-\ve} \left( f(1)+\sqrt{f^2(1 ) - 4}\right)^{\ve} }{ 2 \dsp
 {\sqrt[4]{(f^2(0)-4)(f^2(1)-4)}}}
\label{dsoe}
\eeq
\end{theo}

\medskip
\begin{rem}
As long as $f(0)\neq f(1)$, the above limit can be adjusted to any positive number just by choosing the correct shift $\ve$.  By Theorem~\ref{altindex2},  the LSD of the eigenvalues of $T_n(f;\ve)$ does not depend on $\ve$.  The formula \eqref{First KMS}   holds for $T_n(f;\ve)$ for any $\ve$:
\[ \lim_{n\rightarrow\infty} \frac{ \text{Tr}[\varphi(T_n(f;\ve ))]}{n} = \frac{1}{2\pi} \int_0^1 \int_0^{2\pi} \varphi(f(x)-2\cos t) \, dt \, dx 
\]
If one scales by $G(f)$, one thus obtains a family of matrices whose asymptotic eigenvalue distribution is invariant, but the determinant can be made to converge to any positive number.
\label{shiftremark}
\end{rem}


The above results can also be modified for the case when $f$ has a finite number of jump discontinuities.

\begin{theo}
(i) Let $f$ be twice differentiable on $[0,1]$, with a bounded second derivative,  except for  $r<\infty$ jump discontinuities at $c_1,\dots, c_r \in (0,1)$, where both sided limits exist and are finite, and $f$ is left-continuous at $c_j$: $f(c_j)=f(c_j-)$.  Suppose, also, that  $f>2+\epsilon$ for some $\epsilon > 0$.  Then

\beq
  \frac{D_n(f)}{G(f)^n} = \alpha \prod_{j=1}^r \beta_j \gamma_j^{\{nc_j\}} + o(1)
  \label{jumpeqn}
\eeq
where $\{x\}=x-\lfloor x \rfloor$ is the fractional part of $x$,
\begin{align*}
\alpha &= \frac{1}{2} \frac{f(1)+\sqrt{f^2(1) - 4}}{\sqrt[4]{(f^2(0)-4)(f^2(1)-4)}} 
\\[.1in]
\beta_j &= \frac{f(c_j-)-f(c_j+) + \sqrt{f^2(c_j+)-4} + \sqrt{f^2(c_j-)-4}}{2\sqrt[4]{(f^2(c_j+)-4)(f^2(c_j-)-4)}}
 \\[.1in]
\mbox{and } \quad \gamma_j &= \frac{f(c_j+)+\sqrt{f^2(c_j+)-4}}{f(c_j-)+\sqrt{f^2(c_j-)-4}}
\end{align*}

\smallskip
\noindent
(ii) If $f$ is right-continuous at $c_j$, then the formula \eqref{jumpeqn} holds with $\{c_jn\}$ replaced by $\{c_jn\}'$, where 
\[ \{x\}' =
1+x-\lceil x\rceil
\]
is the fractional part of $x$, but equal to $1$ if $x$ is an integer.
\label{jumptheorem}
\end{theo}

\begin{rem}
Note that $\beta_j$ and $\gamma_j$ are $1$ if $f$ is continuous at $c_j$, so \eqref{jumpeqn} reduces to \eqref{dsoe} for $\ve=1$ when $f$ is twice differentiable.  Since $\{ c_jn\} =\{c_jn\}'$ if $c_j$ is irrational,  the difference between cases (i) and (ii) of the above theorem only occurs when $c_j$ is rational.  In that case the difference arises when $f$ is evaluated at the point $c_j$.  
\end{rem}

\begin{rem}
Obviously, if there is a discontinuity in $f$, the limit
\[ \lim_{n\rightarrow\infty} \frac{D_n(f)}{G(f)^n}
\]
does not exist.  However, we can calculate the $\limsup$ and $\liminf$.  For example, if there is one jump discontinuity at  $c=p/q$,
\begin{align*}
\limsup_{n\rightarrow\infty} \frac{D_n(f)}{G(f)^n} &= \alpha\cdot \beta\cdot \max\{ \gamma^{1/q}, \gamma\}
\\
\liminf_{n\rightarrow\infty} \frac{D_n(f)}{G(f)^n} &= \alpha\cdot \beta \cdot \min\{ \gamma^{1/q}, \gamma\}
\end{align*}
If $c$ is irrational, the same is true with $\gamma^{1/q}$ replaced by $1$.  Analogous statements hold when there are $r$ jump discontinuities.
\end{rem}

To illustrate the asymptotic behavior of $D_n(f)/G(f)^n$, we
consider the case of a single jump discontinuity at $c\in (0,1)$.
If $c =p/q$ is rational,  $\{D_n(f)/G(f)^n\}$ (modulo an $o(1)$ term) is cyclic of order $q$.    When $c$ is irrational, $\{D_n(f)/G(f)^n\}$ is dense on the interval  between $\alpha\beta$ and $\alpha\beta\gamma$.  This is another indication of how exquisitely sensitive  $D_n(f)/G(f)^n$ is.  The slightest irrational perturbation of the point of discontinuity from $c=1/2$, causes the values of $D_n(f)/G(f)^n$ (modulo the $o(1)$ term) to go from alternating between two values to taking on infinitely many values.  This behavior is illustrated in figure~\ref{jumpfig}.  There we calculate $D_n(f)/G(f)^n$ for the piecewise function
\beq 
f(x) = \begin{cases} 3+x^2 +\sqrt{x} \sin(13 x) & x < c \\
4.5-\cos(20x)/x & x \geq c \end{cases}
\label{ff}
\eeq
We compare the values of $D_n(f)/G(f)^n$ with $\alpha\beta\gamma^{\{cn\}'}$ in the case when $c$ is rational and $c$ is irrational.  Agreement is quite good for moderately large $n$.

\begin{figure}[h]
\begin{center}
%
\includegraphics*[width=2.3in]{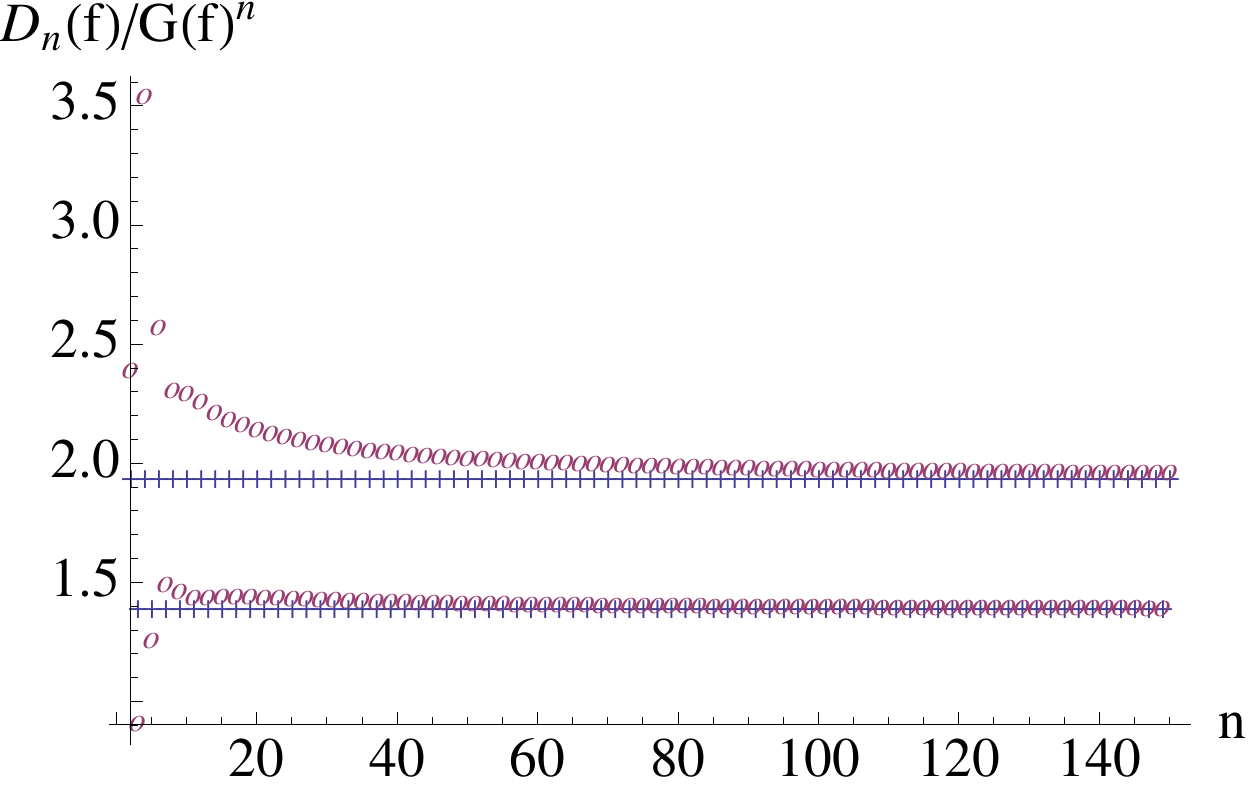}
\includegraphics*[width=2.3in]{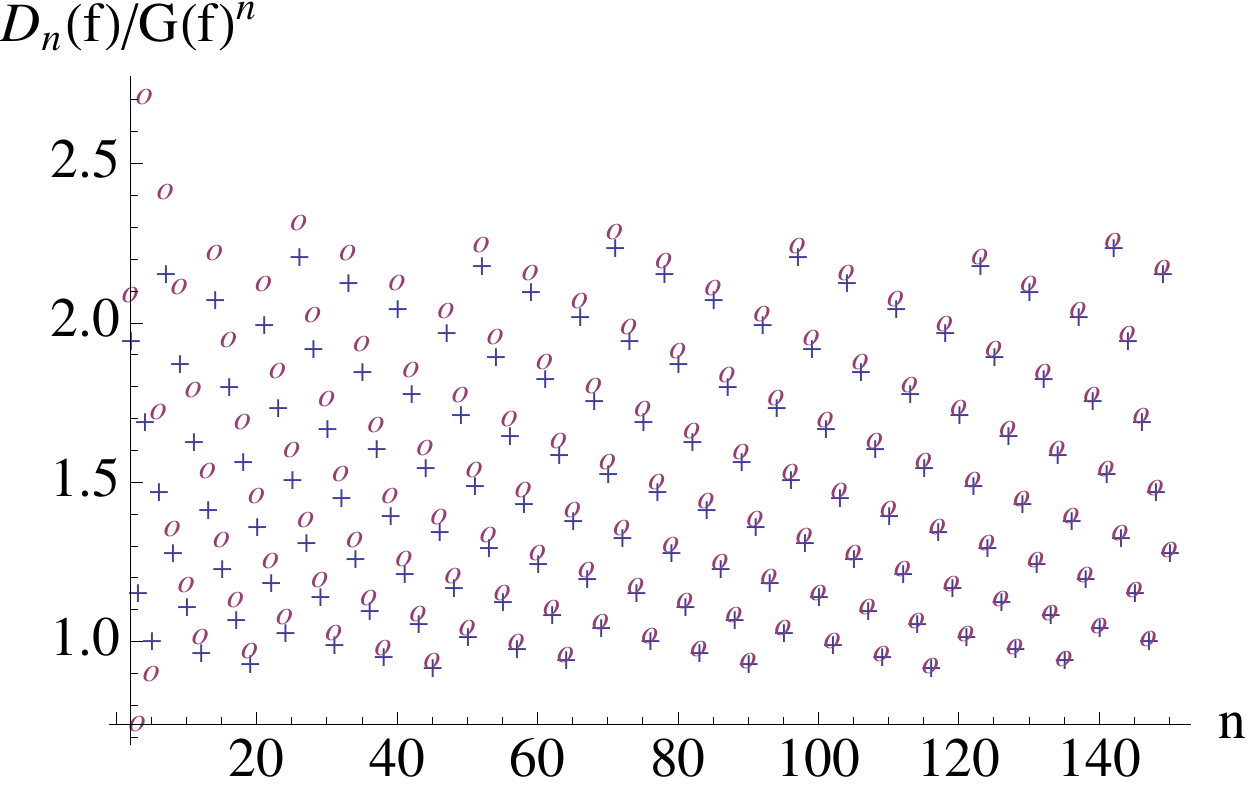}
\end{center}
\caption{$n$ vs. $D_n(f)/G(f)^n$  for $f$ as in \eqref{ff}.  Left: $c=1/2$; right: $c=.9-1.5/\pi\approx .4225$.   The values of $D_n(f)/G(f)^n$ are marked with circles;  the values of $\alpha\beta\gamma^{\{cn\}'}$ are marked with $+$'s.  }
\label{jumpfig}
\end{figure}

\section{Some open problems, conjectures and other remarks}
\label{conclusion}

It is perhaps not surprising that Szeg\H{o}'s limit theorems can be generalized to KMS matrices.  These matrices begin to look very much like Toeplitz matrices, locally, when $n$ is large.  There are a whole host of problems revolving around generalizing other results for Toeplitz matrices to KMS matrices.  While some results for Toeplitz matrices have natural generalizations to KMS matrices, there are some challenges for KMS matrices that are fundamentally different from their Toeplitz counterparts.  

\smallskip

The most powerful results for KMS matrices are for when the matrices are normal.  In this case, we can derive the LSD.  In effect, we can completely characterize how the spectrum behaves in the limit of large $n$.  When $T_n(a)$ is not normal, all we can do is delimit a region in the complex plane where most of the eigenvalues will lie.  And, as figure~\ref{lame_ex} shows, this estimate may not give much detail for the spectrum.
  Normal matrices also have other desirable properties such as numerical stability that are lacking in non-normal matrices.  For these reasons, it is of interest to be able to characterize when a matrix, or sequence of matrices, is normal.  For KMS matrices, it is easy to tell if they are Hermitian:  $T_n(a)$ is Hermitian if and only of the symbol $a(x,t)$ is real valued.  A condition along these lines for normality would be quite valuable.
  
\smallskip

By a result of Brown  and Halmos \cite{brha64} (see also \cite{nore09,nore11}), a Toeplitz matrix is normal if and only if its symbol $b$ can be obtained by rotation and translation of a real-valued symbol $a$, i.e. $b(t)=c+e^{i\theta} a(t)$.  (It follows that the spectrum of a normal Toeplitz matrix lies on a line in the complex plane.)
However, no such characterization for KMS matrices has been found.   In fact, we were unable to construct sequences of normal KMS matrices outside the obvious Toeplitz and Hermitian ones.  This leads to our first problem:

\bprob
Characterize normal KMS matrices.  Give a condition on $a(x,t)$ to guarantee that $T_n(a)$ will be normal.
\eprob

It has long been known that some non-normal Toeplitz matrices are so-called \textit{canonically distributed}, as Widom calls them.  This means that the spectra accumulate on the range of the symbols.  The question of which Toeplitz matrices are canonically distributed is tied to the types of singularities in the symbol \cite{wi90, wi94}.
Thus, another avenue of research  is to study asymptotics when  the symbol $a(x,t)$ has a singularity in $t$.  The famous Fisher-Hartwig conjecture was recently solved for Toeplitz matrices \cite{deitkr11}.  The first step is the following:
\bprob
Formulate a generalized Fisher-Hartwig conjecture for KMS matrices.
\eprob
\noindent
The second step, obviously, is to prove the conjecture.

\smallskip

In the general non-normal, non canonically distributed case, the spectrum of KMS matrices still appears to have some structure.  Schmidt and Spitzer \cite{scsp60} showed that if $T_n(a)$ is Toeplitz where $a(t)$ is a trigonometric polynomial, then the  spectrum of $T_n(a)$ accumulates on a set $\Lambda(a)$ consisting of the union of a finite number of pairwise disjoint open analytic arcs and a finite number of exceptional points (branch points and points $\lambda$ such that for some open neighborhood $U$ of $\lambda$, $\Lambda(a)\cap U$ is an analytic arc starting and terminating on $\partial U$). Ullman \cite{ul67} proved that $\Lambda(a)$ is connected.
Based on numerical experiments such as that seen in figure~\ref{cluster_ex},
 we conjecture that the same holds for KMS matrices when the symbol is smooth in $x$.  Hirschmann \cite{hi67} obtained an implicit formula for the asymptotic density of eigenvalues for banded Toeplitz matrices.  It may be possible to do the same for KMS matrices:

\bprob
Determine the LSD for non-normal KMS matrices.
\eprob

\smallskip

Necessary conditions for the first and second theorems for KMS matrices are still lacking.
For the first theorem, it was Trotter who first proved that it was enough for the $\hat{a}_k$'s to be Riemann integrable.  In their paper on the Jacobi case, Kuijlaars and Serra Capizzano \cite{kuse01} reproved this result for tridiagonal KMS matrices using potential theory.  Their result was for coefficients satisfying a closeness condition for functions only in $L^1([0,1])$.  Clearly, it is not sufficient to take the $\hat{a}_k(x)$ to be only in $L^1([0,1])$.  However, it may be enough for the $\hat{a}_k$ to be in $L^1([0,1])$, with the condition that the set of discontinuities is nowhere dense.

\bprob
Determine necessary conditions on the $\hat{a}_k$ for the first theorem to hold.
\eprob

Consider the banded case where the symbol $a$ is a trigonometric polynomial 
\[ a(x,t) = \sum_{k=p}^q \hat{a}_k(x) e^{ikt}
\]
Then, by Corollary~\ref{STtrig}, as long as the $\hat{a}_k$ are $C^{1+\epsilon}$, 
the limit 
\[ \lim_{n\rightarrow\infty}\frac{\det T_n(a)}{G(a)^{n+1}}
\]
exists and can be calculated.  However, it is not known if this condition is necessary.  In figure~\ref{DSOfig} below we show the fraction $D_n(f)/G(f)^{n}$ for the discrete Schr\"{o}dinger operator \eqref{DSOdef} where $f$ has decreasing levels of regularity.  Numerical evidence suggests that it is not sufficient for $f(x)$ to be merely $C^1$, although we have no proof of this result.

\begin{figure}[h]
\begin{center}
\includegraphics*[width=1.6in]{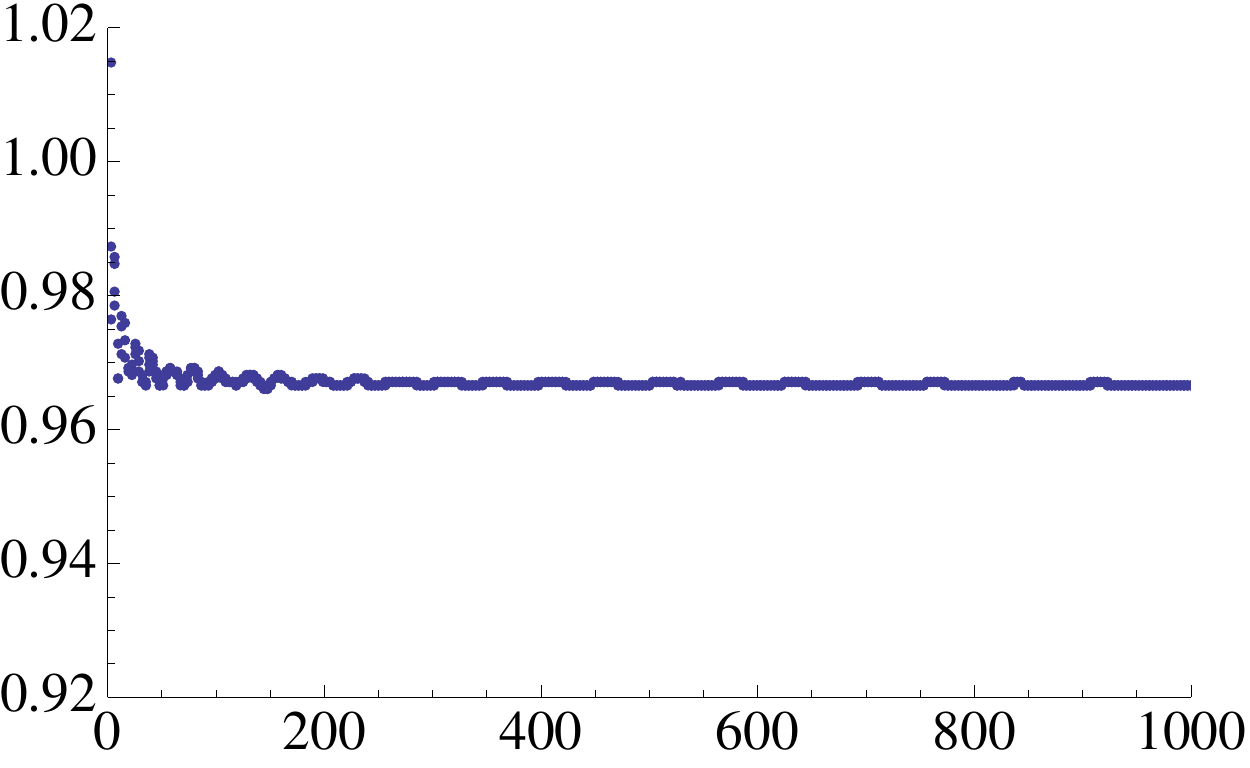}
\includegraphics*[width=1.6in]{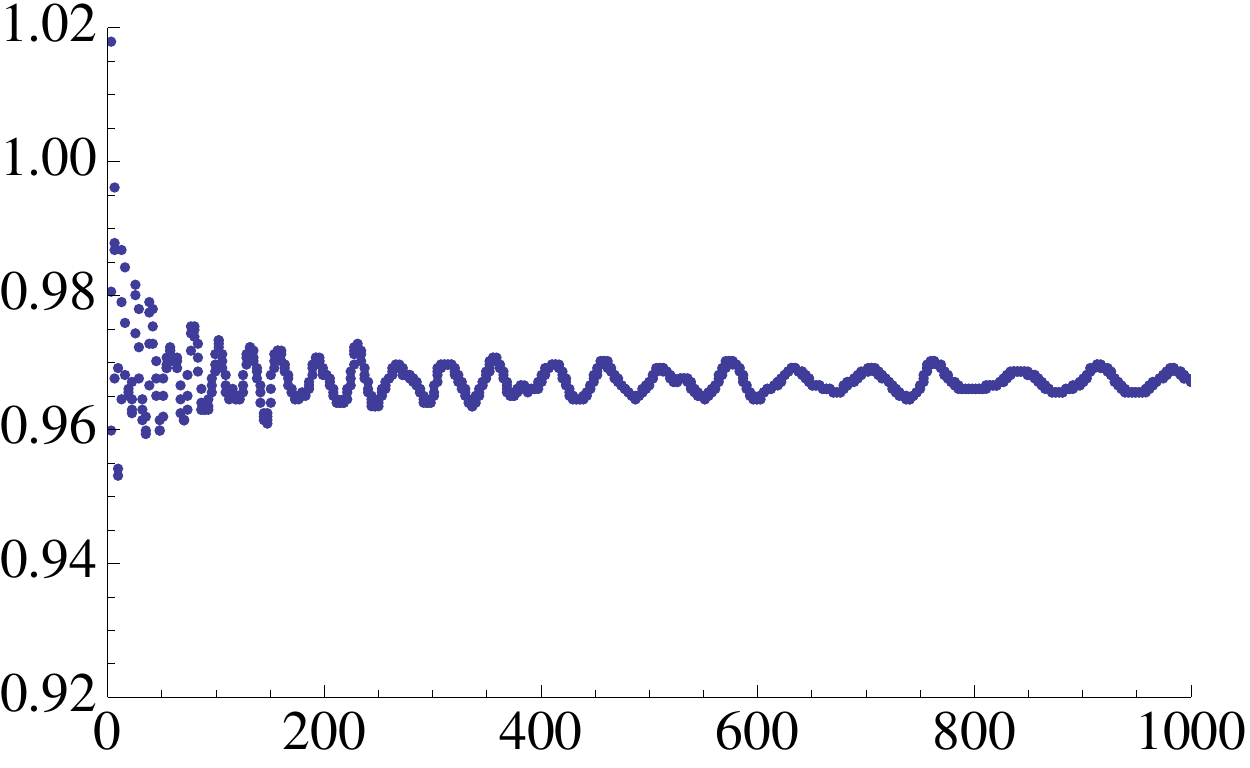}
\includegraphics*[width=1.6in]{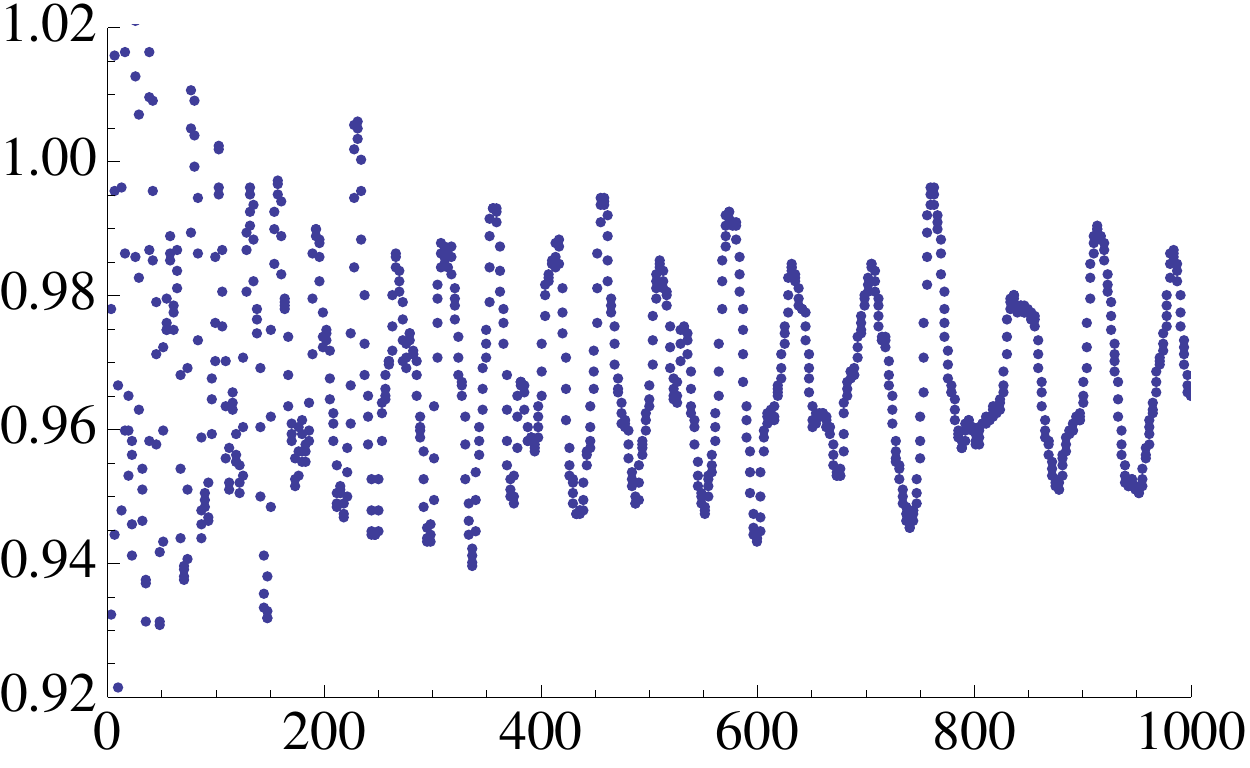}
\end{center}
\caption{$\dsp \frac{D_n(f)}{G(f)^{n}}$ for the discrete Schr\"{o}dinger operator \eqref{DSOdef}.  Left: $f(x)=3.5+x^2\sin(1/x)$; middle: $f(x)=3.5+x^{3/2}\sin(1/x)$; right: $f(x)=3.5+x\sin(1/x)$.}
\label{DSOfig}
\end{figure}

\bprob
Determine necessary and sufficient conditions on the $\hat{a}_k$ for the second theorem to hold.
\eprob

One of the shortcomings of the Strong Szeg\H{o} Theorem is that the error term is computed in terms of a series of the Fourier coefficients of the logarithm.  It can thus be somewhat opaque to deduce properties of the error based on the symbol itself. The only case where we have a simple formula for the error is  for the discrete Schr\"{o}dinger operator:  the formula \eqref{dsoe} in the continuous case, and \eqref{jumpeqn} when there are jump discontinuities. It may be possible to use similar methods to obtain results for other operators. 

\bprob
Extend the formulas \eqref{dsoe} and \eqref{jumpeqn} to tridiagonal KMS matrices. 
\eprob

For Toeplitz band matrices, the $o(1)$ error term in \eqref{Strong Szego1} is in fact $\Oo(q^n)$ for some $0<q<1$ (see e.g. \cite{bogr05}). Consequently, Szeg\H{o}'s Strong Theorem gives the complete asymptotic expansion of $ \det(T_n(a))$ when $a$ is a trigonometric polynomial. All of the existing proofs of the strong theorem for KMS matrices make use of the Euler-Maclaurin formula, and hence the $o(1)$ error term is the best known result. 

\bprob
Improve the error term in the strong theorem for KMS band matrices. 
\eprob

In applications (e.g. in determining stability) one is often interested in the bounds for the eigenvalues.  There is still the open problem:
\bprob
Determine the extreme eigenvalues of a sequence of KMS matrices.
\eprob

The present paper is concerned with the spectra of KMS matrices.  However, in applications it is often more important to know the pseudospectra.  Also of great importance are the eigenvectors of KMS matrices, of which we have said nothing.  While a great deal is known about the pseudospectra of Toeplitz matrices (see \cite{bogr05} and many references therein), not much is known about the pseudospectra of KMS matrices.  The paper of Trefethan and Chapman \cite{trch04} contains some interesting results on the pseudospectra and eigenvectors of KMS (called ``twisted Toeplitz'' in that paper) matrices.  However, there is much more to be discovered.  We can thus pose the somewhat general problem:

\bprob
Derive bounds on the pseudospectra of KMS matrices. 
\eprob

There is by now an enormous, and growing, literature on Toeplitz matrices.  Only a few of the results for Toeplitz matrices, importantly the first and second Szeg\H{o}'s theorems, have been extended to KMS matrices.  We look forward with great anticipation to the development of theory for these natural generalizations.

\end{document}